\documentclass[a4,12pt]{amsart}

\usepackage{amssymb}
\usepackage{amstext}
\usepackage{amsmath}
\usepackage{amscd}
\usepackage{amsthm}
\usepackage{amsfonts}
\usepackage{enumerate}
\usepackage{graphicx}
\usepackage{latexsym}
\usepackage[all]{xy}
\usepackage[usenames]{color}

\newtheorem*{corollary*}{Corollary}
\newtheorem{theorem}{Theorem}[section]
\newtheorem{corollary}[theorem]{Corollary}
\newtheorem{lemma}[theorem]{Lemma}
\newtheorem{proposition}[theorem]{Proposition}
\newtheorem*{proposition*}{Proposition}
\newtheorem{question}[theorem]{Question}

\newtheorem*{claim*}{Claim}

\theoremstyle{definition}
\newtheorem{definition}[theorem]{Definition}
\newtheorem{remark}[theorem]{Remark}
\newtheorem{example}[theorem]{Example}
\newtheorem*{notation}{Notation}

\newtheorem{definitiontheorem}[theorem]{Definition-Theorem}
\newtheorem{method}[theorem]{Method}
\newtheorem{problem}[theorem]{Problem}
\newtheorem*{acknowledgement}{{\bf Acknowledgement}}

\theoremstyle{remark}

\numberwithin{equation}{theorem}
\renewcommand{\mod}{\operatorname{mod}}

\newcommand{\smod}{\operatorname{\underline{mod}}}
\newcommand{\End}{\operatorname{End}}
\newcommand{\Hom}{\operatorname{Hom}}

\newcommand{\Ext}{\operatorname{Ext}}
\newcommand{\sHom}{\operatorname{\underline{Hom}}}

\newcommand{\Db}{\mathsf{D}^{\rm b}}

\newcommand{\m}{\mathsf{m}}

\newcommand{\C}{\mathsf{C}}

\newcommand{\mult}{\operatorname{\mathsf{mult}}}

\newcommand{\na}{\operatorname{\mathsf{na}}}
\newcommand{\vx}{\operatorname{\mathsf{vx}}}
\begin{document}
\title{Derived equivalences between symmetric special biserial algebras}
\date{\today}
\author{Takuma Aihara}
\address{Fakult\"at f\"ur Mathematik, Universit\"at Bielefeld, D-33501 Bielefeld, Germany}
\curraddr{Graduate School of Mathematics, Nagoya University, Furocho, Chikusaku, Nagoya 464-8602, Japan}
\email{aihara.takuma@math.nagoya-u.ac.jp}
\begin{abstract}
The notion of mutation plays crucial roles in representation theory of algebras.
Two kinds of mutation are well-known: tilting/silting mutation and quiver-mutation.
In this paper, we focus on tilting mutation for symmetric algebras.
Introducing mutation of SB quivers, 
we explicitly give a combinatorial description of tilting mutation of symmetric special biserial algebras.
As an application, we generalize Rickard's star theorem.
We also introduce flip of Brauer graphs and apply our results to Brauer graph algebras.
\end{abstract}
\maketitle

\section{Introduction}

Tilting mutation, which is a special case of silting mutation \cite{AI},
was introduced by Riedtmann-Schofield \cite{RS} and Happel-Unger \cite{HU}
to investigate the structure of the derived category.
For example, Bernstein-Gelfand-Ponomarev reflection functors \cite{BGP}, 
Auslander-Platzeck-Reiten tilting modules \cite{APR} and Okuyama-Rickard tilting complexes \cite{O,R2}
are special cases of tilting mutation.
In the case that a given algebra is symmetric, 
tilting mutation yields infinitely many tilting complexes.
From Morita theoretic viewpoint of derived categories \cite{R1}, 
tilting complexes are extremely important complexes 
since they give rise to derived equivalences which preserve many homological properties.

%


The following problem is naturally asked:
\begin{problem}\label{problem}
Give an explicit description of the endomorphism algebra of a tilting complex given by tilting mutation.
\end{problem}

In this paper we give a complete answer to this problem for symmetric special biserial algebras,
which is one of the important classes of algebras in representation theory.
Some of special biserial algebras were first studied by Gelfand-Ponomarev \cite{GP},
and also naturally appear in modular representation theory of finite groups \cite{Al,E}.
Moreover such an algebra is always representation-tame and 
the classification of all indecomposable modules of such an algebra was provided in \cite{WW,BR}.
The derived equivalence classes of special biserial algebras were also discussed in \cite{AA, BHS,K,KR, SZ}.

To realize our goal,
we start with describing symmetric special biserial algebras in terms of combinatorial data,
which we call \emph{SB quivers}.
Moreover we will study symmetric special biserial algebras from graph theoretic viewpoint,
which is discribed by \emph{Brauer graphs}.
Indeed, we have the result below (see Lemma \ref{cycle decomposition} and Proposition \ref{SSBA is BGA}):

\begin{proposition}\label{1-1}
There exist one-to-one correspondences among the following three classes:
\begin{enumerate}[{\rm (1)}]
\item Symmetric special biserial algebras;
\item Special quivers with cycle-decomposition (\emph{SB quivers});
\item Brauer graphs.
\end{enumerate}
\end{proposition}

%

%


We introduce \emph{mutation of SB quivers} 
(see Definition \ref{QM}, \ref{QM multiplex} and \ref{QM multiplex 2}),
which is similar to 
Fomin-Zelevinsky quiver-mutation \cite{FZ}.
Moreover we will show that mutation of SB quivers corresponds to a certain operation on Brauer graphs,
which we call \emph{flip}
and is a generalization of mutation/flip of Brauer trees introduced in \cite{A2,KZ}.

The main theorem of this paper is the following:

\begin{theorem}[Theorem \ref{compatibility} and Theorem \ref{main Brauer graph}]\label{main1}
The following three operations are compatible each other:
\begin{enumerate}[{\rm (1)}]
\item Tilting mutation of symmetric special biserial algebras;
\item Mutation of SB quivers;
\item Flip of Brauer graphs.
\end{enumerate}
\end{theorem}

We note that certain special cases of the compatibility of (1) and (3) in Theorem \ref{main1}
were given by \cite{K,An} (see Remark \ref{remark of Kauer}). 

The strategy of our proof of Theorem \ref{main1} is, 
by using the fact that the property of being a symmetric special biserial algebra is derived invariant,
to resolve into the calculation of numerical invariants but not directly calculate endomorphism algebras.
Thus, our proof is simpler than that of \cite{K, An}.

As an application of Theorem \ref{main1}, we generalize ``Rickard's star theorem'' for Brauer tree algebras,
which gives nice representatives of Brauer tree algebras up to derived equivalence \cite{R2, MH}.
We introduce \emph{Brauer double-star algebras}, as the corresponding class for Brauer tree algebras (see also \cite{K,KR,Ro}),
and prove the following (see Section \ref{star theorem} for the details):

\begin{theorem}[Theorem \ref{reduction of Brauer graph}]\label{main2}
Any Brauer graph algebra is derived equivalent to a Brauer double-star algebra
whose Brauer graph has the same number of the edges and the same multiplicities of the vertices.
\end{theorem}

Our proof of Theorem \ref{main2} is given in terms of the corresponding symmetric special biserial algebras 
(Theorem \ref{reduction theorem}).
As an application of Theorem \ref{main2}, we deduce Rickard's star theorem (Corollary \ref{Rickard reduction theorem}).


This paper is organized as follows:
In Section \ref{section of SSBA}, 
we recall the definition of special biserial algebras and introduce the notion of SB quivers. 
It is seen that the class of symmetric special biserial algebras coincides with 
that of SB quivers (Lemma \ref{cycle decomposition}),
which plays an important role in this paper.
%
Moreover, we study tilting complexes introduced by Okuyama and Rickard. 
In Section \ref{section of mutation},
mutation of SB quivers is introduced.
Theorem \ref{main1} is also stated.
In Section \ref{section of reduction},
applying Theorem \ref{main1}, we establish a method for reducing some cycles (Method \ref{reduction method}).
It is observed that this method implies `reduction' theorem for symmetric special biserial algebras.
In Section \ref{Brauer graph},
we introduce flip of Brauer graphs and 
see that it is compatible with tilting mutation of Brauer graph algebras (Theorem \ref{main Brauer graph}).
Moreover, Theorem \ref{main2} is also obtained.
In Section \ref{proof},
we shall prove Theorem \ref{main1}.

\begin{acknowledgement}
The author would like to express his deep gratitude to Osamu Iyama
who read the paper carefully and gave a lot of helpful comments and suggestions.
The author would also like to thank the referees for carefully reading the paper and for giving constructive comments.
\end{acknowledgement}


\section{Symmetric special biserial algebras}\label{section of SSBA}

This section is devoted to introducing the notion of SB quivers.
We will give a relationship between symmetric special biserial algebras and SB quivers.
Moreover we study tilting mutation,
which is a special case of silting mutation introduced by \cite{AI}. 

Throughout this paper, we use the following notation.

\begin{notation}
Let $A$ be a finite dimensional algebra over an algebraically closed field $k$. 
\begin{enumerate}[(1)]
\item We always assume that $A$ is basic and indecomposable.
\item We often write $A=kQ/I$ where $Q$ is a finite quiver with relations $I$. 
The sets of vertices and arrows of $Q$ are denoted by $Q_0$ and $Q_1$, respectively.
\item We denote by $\mod A$ the category of finitely generated right $A$-modules. 
A simple (respectively, indecomposable projective) $A$-module corresponding to a vertex $i$ of $Q$
is denoted by $S_i$ (respectively, by $P_i$). 
We always mean that a module is finitely generated. 
\end{enumerate}
\end{notation}

A quiver of the form 
$\xymatrix{
\bullet \ar[r] & \bullet \ar[r] & \cdots \ar[r] & \bullet \ar@/_1pc/[lll]_{}
}$
with $n$ arrows is called an $n$-\emph{cycle} (for simplicity, \emph{cycle}).
We mean 1-cycle by \emph{loop}. 

Let us start with introducing SB quivers.

\begin{definition}
We say that a finite connected quiver $Q$ is \emph{special} if any vertex $i$ of $Q$ is 
the starting point of at most two arrows and also the end point of at most two arrows.
For a special quiver $Q$ with at least one arrow, a set $\C=\{C_1, C_2,\cdots,C_v\}$ of cycles in $Q$
with a function $\mult:\C\to\mathbb{N}$ is said to be
a \emph{cycle-decomposition} if it satisfies the following conditions:
\begin{enumerate}[(1)]
\item Each $C_\ell$ is a subquiver of $Q$ with at least one arrow such that 
$Q_0=(C_1)_0\cup\cdots\cup (C_v)_0$ and $Q_1=(C_1)_1\amalg\cdots\amalg (C_v)_1$:
For any $\alpha\in Q_1$, we denote by $C_\alpha$ a unique cycle in $\C$ which contains $\alpha$.
\item Any vertex of $Q$ belongs to at most two cycles.
\item $\mult(C_\ell)>1$ if $C_\ell$ is a loop.
\end{enumerate} 
A \emph{SB quiver} is a pair $(Q,\C)$ of a special quiver $Q$ and its cycle-decomposition $\C$.
\end{definition}

Let $(Q,\C)$ be a SB quiver.
For each cycle $C$ in $\C$,
we call $\mult(C)$ the \emph{multiplicity} of $C$.
For any arrow $\alpha$ of $Q$, we denote by $\na(\alpha)$ a unique arrow $\beta$ such that
$\xymatrix{\bullet \ar[r]^\alpha & \bullet \ar[r]^\beta & \bullet}$ appears in $C_\alpha$.

We construct a finite dimensional algebra from a SB quiver.

\begin{definition}
Let $(Q,\C)$ be a SB quiver. 
An ideal $I_{(Q,\C)}$ of $kQ$ is generated by the following three kinds of elements:
\begin{enumerate}[(i)]
\item $(\alpha_t\alpha_{t+1}\cdots\alpha_{t+s-1})^m\alpha_t$ for each cycle $C$ in $\C$ of the form 
\[\xymatrix{i_1 \ar[r]^{\alpha_1} & i_2 \ar[r]^{\alpha_2} & \cdots \ar[r]^{\alpha_{s-1}} & i_s \ar@/^1pc/[lll]^{\alpha_s}}\]
and $t=1,2,\cdots,s$,
where $m=\mult(C)$ and the indices are considered in modulo $s$.
\item $\alpha\beta$ if $\beta\neq\na(\alpha)$.
\item $(\alpha_1\alpha_2\cdots\alpha_s)^m-(\beta_1\beta_2\cdots\beta_t)^{m'}$ 
whenever we have a diagram
\[\xymatrix@R=1ex{
\cdots \ar[r] & i_s \ar[rd]^{\alpha_s} & & i'_t \ar[ld]_{\beta_t} & \cdots \ar[l] \\
&& i \ar[ld]^{\alpha_1} \ar[rd]_{\beta_1} && \\
\cdots & i_2 \ar[l]_{\alpha_2} & & i'_2 \ar[r]^{\beta_2} & \cdots
}\]
where $C_{\alpha_\ell}=C_{\alpha_1}, C_{\beta_{\ell'}}=C_{\beta_1}$ for any $1\leq\ell\leq s,1\leq \ell'\leq t$
and $m=\mult(C_{\alpha_1}), m'=\mult(C_{\beta_1})$.
\end{enumerate}
We define a $k$-algebra $A:=A_{(Q,\C)}$ associated with $(Q,\C)$ by $A=kQ/I_{(Q,\C)}$. 
Then the algebra $A_{(Q,\C)}$ is finite dimensional and symmetric.
The cycle-decomposition $\C$ is also said to be the \emph{cycle-decomposition} of $A_{(Q,\C)}$. 
\end{definition}

An algebra $A:=kQ/I$ is said to be \emph{special biserial} if $Q$ is special and 
for any arrow $\beta$ of $Q$, there is at most one arrow $\alpha$ with $\alpha\beta\not\in I$
and at most one arrow $\gamma$ with $\beta\gamma\not\in I$.

Thanks to \cite{Ro,An} (see Proposition \ref{SSBA is BGA}), we have the following result.

\begin{lemma}\label{cycle decomposition}
The assignment $(Q,\C)\mapsto A_{(Q,\C)}$ gives rise to a bijection between 
the isoclasses of SB quivers and those of symmetric special biserial algebras. 
\end{lemma}


\begin{example}\label{examples of ssb}
\begin{enumerate}[(1)]
\item Let $Q$ be the quiver 
\[\xymatrix{
&1 \ar@<0pc>[dl]^\alpha \ar@<0.4pc>[dr]^{\gamma'}& \\
2 \ar@<0.4pc>[ur]^{\alpha'} \ar@<0.2pc>[rr]^\beta & & 3 \ar@<0pc>[ul]^\gamma \ar@<0.2pc>[ll]^{\beta'}
}\]
with the relations $I:=\langle\alpha\beta, \beta\gamma, \gamma\alpha, \alpha'\gamma', \gamma'\beta',\beta'\alpha',
\alpha'\alpha-\beta\beta', \beta'\beta-\gamma\gamma', \gamma'\gamma-\alpha\alpha'\rangle$.
Then the algebra $A:=kQ/I$ is symmetric special biserial associated with 
the SB quiver $(Q,\C)$
where the cycle-decomposition is 
\[\C=\left\{
\left(\xymatrix{1 \ar@<0.2pc>[r]^{\alpha} & 2\ar@<0.2pc>[l]^{\alpha'}}\right),
\left(\xymatrix{2 \ar@<0.2pc>[r]^{\beta} & 3\ar@<0.2pc>[l]^{\beta'}}\right), 
\left(\xymatrix{3 \ar@<0.2pc>[r]^{\gamma} & 1\ar@<0.2pc>[l]^{\gamma'}}\right)
\right\}
\]
such that the multiplicity of every cycle is 1. 

\item Let $Q$ be the quiver 
\[\xymatrix{
& 1 \ar@(ru,ul)_\delta \ar[dl]_\alpha & \\
2 \ar[rr]_\beta && 3 \ar[lu]_\gamma 
}\]
with the relations $I:=\langle\gamma\alpha, (abcd)^2a\ |\ \{a,b,c,d\}=\{\alpha,\beta,\gamma,\delta\} \rangle$. 
Then $A:=kQ/I$ is a symmetric special biserial algebra which is isomorphic to $A_{(Q,\C)}$,
where $\C$ is the cycle-decomposition 
\[\C=\left\{\left(\begin{array}{c}\xymatrix{
1 \ar[r]^\alpha & 2 \ar[d]^\beta \\
1 \ar[u]^\delta & 3 \ar[l]^\gamma
}\end{array}\right)\right\}\]
with the multiplicity 2.

\item Let $Q$ be the quiver 
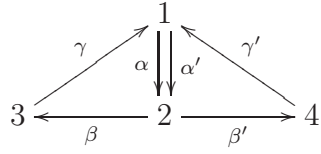
\[\xymatrix@C=1.5cm{
& 1 \ar@<-0.2pc>[d]_\alpha \ar@<0.2pc>[d]^{\alpha'}  & \\
3 \ar[ru]^\gamma & 2 \ar[l]^\beta \ar[r]_{\beta'} & 4 \ar[ul]_{\gamma'} 
}\]
with cycle-decomposition $\C=\{C_1, C_2\}$
where 
\[C_1=\left(\begin{array}{c}
\xymatrix@R=0.5cm{
 & 1 \ar[d]^\alpha \\
3 \ar[ur]^\gamma & 2 \ar[l]^\beta
}
\end{array}\right),
C_2=\left(\begin{array}{c}
\xymatrix@R=0.5cm{
1 \ar[d]_{\alpha'} \\
2 \ar[r]_{\beta'} & 4 \ar[ul]_{\gamma'} 
}
\end{array}\right)
\]
and $\mult(C_1)=\mult(C_2)=1$.
Then we have an isomorphism $A_{(Q,\C)}\simeq kQ/I$ 
where $I=\langle \alpha\beta', \alpha'\beta, \gamma\alpha', \gamma'\alpha, (abc)a\ |\ 
\{a,b,c\}=\{\alpha,\beta,\gamma\}, \{\alpha',\beta', \gamma'\} \rangle$.
\end{enumerate}
\end{example}

We know that the property of being symmetric special biserial is derived invariant. 

\begin{proposition}\label{ssb under derived}
Let $A$ and $B$ be finite dimensional algebras. 
Suppose that $A$ and $B$ are derived equivalent. 
If $A$ is a symmetric special biserial algebra,
then so is $B$.
\end{proposition} 
\begin{proof}
Combine \cite{R1} and \cite{P}.
\end{proof}


Next, we recall the notion of tilting mutation.
We refer to \cite{AI} for details.

The bounded derived category of $\mod A$ is denoted by $\Db(\mod A)$. 

We give the definition of tilting complexes. 

\begin{definition}
Let $A$ be a finite dimensional algebra. 
We say that a bounded complex $T$ of finitely generated projective $A$-modules is \emph{tilting} 
if it satisfies $\Hom_{\Db(\mod A)}(T,T[n])=0$ for any integer $n\neq0$ and 
produces the complex $A$ concerned in degree 0 by taking direct summands, mapping cones and shifts. 
\end{definition}

The following result shows the importance of tilting complexes.

\begin{theorem}\cite{R1}
Let $A$ and $B$ be finite dimensional algebras.
Then $A$ and $B$ are derived equivalent if and only if 
there exists a tilting complex $T$ of $A$ such that $B$ is Morita equivalent to 
the endomorphism algebra $\End_{\Db(\mod A)}(T)$.
\end{theorem}

For each vertex $i$ of $Q$, we denote by $e_i$ the corresponding primitive idempotent of $A$. 

We recall a complex given by Okuyama and Rickard \cite{O, R2},
which is a special case of tilting mutation (see \cite{AI}). 

\begin{definitiontheorem}\label{OR complex}\cite{O}
Fix a vertex $i$ of $Q$. 
We define a complex by 
\[T_{j}:=\begin{cases}
\begin{array}{c}
\xymatrix@!R=0.5pt{
 (0{\rm{th}}) & (1{\rm{st}}) & \\
 P_{j} \ar[r] & 0 & (j\neq i)  \\
 P \ar[r]^{\pi_{i}} & P_{i} & (j=i)
}
\end{array}
\end{cases}
\]
where $P\xrightarrow{\pi_{i}}P_{i}$ is a minimal projective presentation of $e_{i}A/e_{i}A(1-e_i)A$. 
Now we call $T(i):=\bigoplus_{j\in Q_0}T_j$ an \emph{Okuyama-Rickard complex} with respect to $i$  
and put $\mu_i^+(A):=\End_{\Db(\mod A)}(T(i))$. 
If $A$ is symmetric, then $T$ is tilting. 
In particular, $\mu_i^+(A)$ is derived equivalent to $A$.
\end{definitiontheorem}

\section{Mutation of SB quivers}\label{section of mutation}

The aim of this paper is to give a purely combinatorial description of tilting mutation of symmetric special biserial algebras.

To do this, we introduce \emph{mutation of SB quivers} by dividing to three cases,
which is a new SB quiver $\mu_i^+(Q,\C)$ made from a given one $(Q,\C)$.

Now, the main theorem in this paper is stated, which gives the compatibility 
between tilting mutation and mutation of SB quivers.
This is proved in Section \ref{proof}.

\begin{theorem}\label{compatibility}
Let $A$ be a symmetric special biserial algebra and take a SB quiver $(Q,\C)$ satisfying $A\simeq A_{(Q,\C)}$.
Let $i$ be a vertex of $Q$.
Then we have an isomorphism $A_{\mu_i^+(Q,\C)}\simeq \mu_i^+(A)$.
In particular, $A_{\mu_i^+(Q,\C)}$ is derived equivalent to $A$.
\end{theorem}


Let $(Q,\C)$ be a SB quiver and $i$ be a vertex of $Q$.
We say that $Q$ is \emph{multiplex} at $i$ if there exists arrows $\xymatrix{i \ar@<0.2pc>[r]^\alpha & j \ar@<0.2pc>[l]^\beta}$
with $\beta\neq\na(\alpha)$ and $\alpha\neq\na(\beta)$.

\subsection{Non-multiplex case}

We introduce mutation of SB quivers at non-multiplex vertices. 

Let $(Q,\C)$ be a SB quiver and fix a vertex $i$ of $Q$.
We define a new SB quiver $\mu_i^+(Q,\C)=(Q',\C')$ as follows.

\subsubsection{Mutation rules}

\begin{definition}\label{QM}
Suppose that $Q$ is non-multiplex at $i$. 
We define a quiver $Q'$ as the following three steps:
\begin{enumerate}[(QM1)]
\item Consider any path 
\[\xymatrix{
h \ar[r]^\alpha & i \ar[r]^(0.25)\beta & j\ \mbox{with } \beta=\na(\alpha) & \mbox{or} & 
h \ar[r]^\alpha & i \ar@(lu,ru)^\gamma \ar[r]^(0.18)\beta & j\ \mbox{ with } \gamma=\na(\alpha), \beta=\na(\gamma)
}\]
for $h\neq i\neq j$.
Then draw a new arrow $\xymatrix{h \ar[r]^x & j}$ 
\begin{enumerate}[(QM1-1)]
\item if $h\neq j$ or
\item[(QM1-2)] if $h=j, \alpha=\na(\beta)\ \mbox{and } \mult(C_\alpha)>1$.
\end{enumerate}
\item Remove all arrows $\xymatrix{h \ar[r] & i}$ for $h\neq i$.
\item Consider any arrow $\xymatrix{i \ar[r]^\alpha & h}$ for $h\neq i$. 
\begin{enumerate}[(QM3-1)]
\item If there exists a path $\xymatrix{i \ar[r]^\alpha & h \ar[r]^\beta & j}$ with $\beta\neq\na(\alpha)$,
then replace it by a new path $\xymatrix{h \ar[r]^x & i \ar[r]^y & j}$.
\item Otherwise, add a new arrow $\xymatrix{h \ar[r]^x & i}$. 
\end{enumerate}
\end{enumerate}
\end{definition}

It is easy to see that the new quiver $Q'$ is again special. 

\subsubsection{Cycle-decompositions}

We give a cycle-decomposition $\C'$ of $Q'$.

\begin{definition}\label{relations of QM}
We use the notation of Definition \ref{QM}.
\begin{enumerate}[(1)]
\item We define a cycle containing a new arrow $x$ in (QM1) as follows: 
\begin{enumerate}[(i)]
\item In the case (QM1-1), $C_x$ is obtained by replacing $\alpha\beta$ or $\alpha\gamma\beta$ in $C_\alpha$ by $x$.
\item In the case (QM1-2), $C_x$ is a new cycle $\xymatrix{h\ar@(ur,dr)^x}$ with multiplicity $\mult(C_\alpha)$.
\end{enumerate}
\item We define a cycle containing a new arrow $x$ and $y$ in (QM3) as follows:
\begin{enumerate}[(i)]
\item In the case (QM3-1), $C_x=C_y$ and replace $\beta$ in $C_\beta$ by $xy$.
\item In the case (QM3-2), 
\begin{enumerate}[(a)]
\item if there exists an arrow $\xymatrix{h \ar[r]^\beta & i}$ of $Q$, 
then $C_x$ is defined by replacing $\beta$ in $C_\beta$ by $x$. 
\item Otherwise, $C_x$ is a new cycle 
\[\begin{cases}
\ \xymatrix{i \ar@<0.2pc>[r]^\alpha & h \ar@<0.2pc>[l]^x} & \mbox{if there is no loop at $i$ belonging to } C_\alpha \\ 
\ \begin{array}{c}\xymatrix@C=0.5cm @R=0.5cm{ & i \ar[dr]^\alpha & \\ i \ar[ur]^\beta &  & h \ar[ll]^x }\end{array} &
\mbox{if there is a loop $\beta$ at $i$ belonging to } C_\alpha
\end{cases}\]
with multiplicity 1.
\end{enumerate}
\end{enumerate}
\end{enumerate}
Then we obtain a cycle-decomposition $\C'$ of $Q'$.
\end{definition}

Thus, we get a new SB quiver $\mu_i^+(Q,\C):=(Q',\C')$, called \emph{right mutation} of $(Q,\C)$ at $i$.

Dually, we define the \emph{left mutation} $\mu_i^-(Q,\C)$ of $(Q,\C)$ at $i$ 
by $\mu_i^-(Q,\C):=\mu_i^+(Q^{\rm op},\C^{\rm op})^{\rm op}$, 
where $Q^{\rm op}$ is the opposite quiver of $Q$ and $\C^{\rm op}$ is the cycle-decomposition of $Q^{\rm op}$ corresponding to $\C$.

\begin{example}\label{examples of QM}
\begin{enumerate}[(1)]
\item Let $(Q,\C)$ be the SB quiver as in Example \ref{examples of ssb} (1).
Then we have the right mutation $\mu_1^+(Q,\C)=(Q',\C')$ of $Q$ at 1 as follows: 
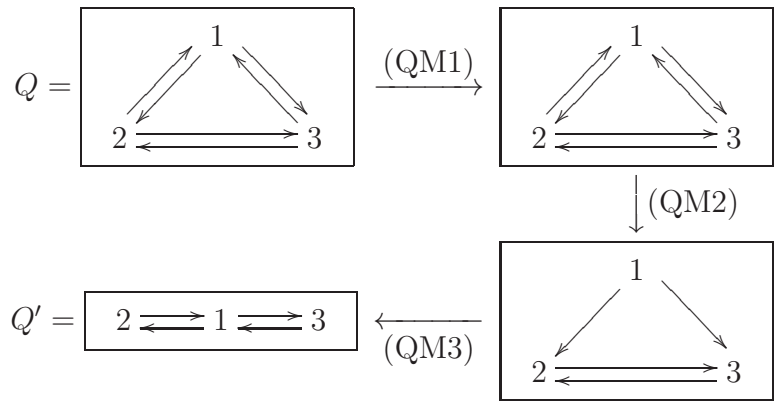
\[\begin{CD}
Q=\fbox{$\begin{array}{c}
\xymatrix{
&1 \ar@<0pc>[dl] \ar@<0.4pc>[dr]& \\
2 \ar@<0.4pc>[ur] \ar@<0.2pc>[rr] & & 3 \ar@<0pc>[ul] \ar@<0.2pc>[ll]
}
\end{array}$} 
@>\mbox{(QM1)}>> 
\fbox{$\begin{array}{c}
\xymatrix{
&1 \ar@<0pc>[dl] \ar@<0.4pc>[dr]& \\
2 \ar@<0.4pc>[ur] \ar@<0.2pc>[rr] & & 3 \ar@<0pc>[ul] \ar@<0.2pc>[ll]
}
\end{array}$} \\
@. @VV\mbox{(QM2)}V \\
Q'=\fbox{$\begin{array}{c}
\xymatrix{
2 \ar@<0.2pc>[r] & 1 \ar@<0.2pc>[r] \ar@<0.2pc>[l] & 3 \ar@<0.2pc>[l] 
}
\end{array}$}
@<<\mbox{(QM3)}< 
\fbox{$\begin{array}{c}
\xymatrix{
&1 \ar@<0pc>[dl] \ar@<0.4pc>[dr]& \\
2 \ar@<0.2pc>[rr] & & 3 \ar@<0.2pc>[ll]
}
\end{array}$}
\end{CD}\]
and 
\[\C'=\left\{
\left(\begin{array}{c}
\xymatrix{
1 \ar[r] & 2 \ar[d] \\
3 \ar[u] & 1 \ar[l]
}
\end{array}\right)
\right\}
\]

\item Let $(Q,\C)$ be the SB quiver of Example \ref{examples of ssb} (2).
Then the right mutation $\mu_1^+(Q,\C)=(Q',\C')$ of $Q$ at 1 is obtained as follows:
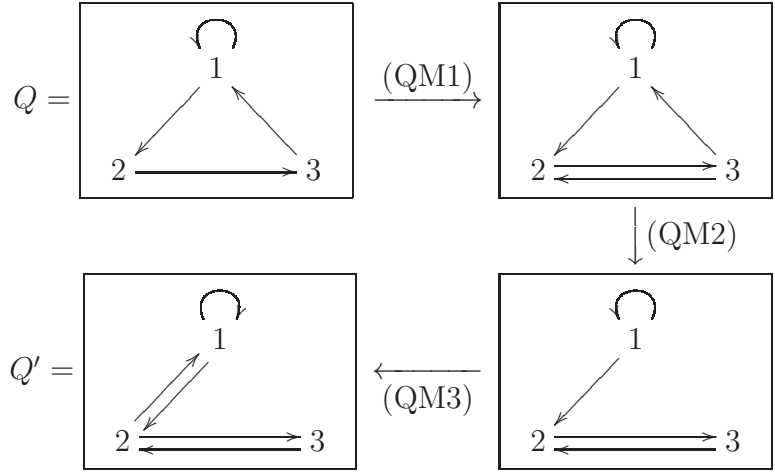
\[\begin{CD}
Q=\fbox{$\begin{array}{c}
\\
\xymatrix{
& 1 \ar@(ru,ul) \ar[dl] & \\
2 \ar[rr] && 3 \ar[lu]
} 
\end{array}$} @>\mbox{(QM1)}>> 
\fbox{$\begin{array}{c}
\\
\xymatrix{
& 1 \ar@(ru,ul) \ar[dl] & \\
2 \ar@<0.2pc>[rr] && 3 \ar[lu] \ar@<0.2pc>[ll]
}
\end{array}$} \\
@. @VV\mbox{(QM2)}V \\
Q'=\fbox{$\begin{array}{c}
\\
\xymatrix{
& 1 \ar@(lu,ur) \ar@<0.2pc>[dl] & \\
2 \ar@<0.2pc>[rr] \ar@<0.2pc>[ur] && 3 \ar@<0.2pc>[ll]
}
\end{array}$} @<<\mbox{(QM3)}<
\fbox{$\begin{array}{c}
\\
\xymatrix{
& 1 \ar@(ru,ul) \ar[dl] & \\
2 \ar@<0.2pc>[rr] && 3 \ar@<0.2pc>[ll]
}
\end{array}$}
\end{CD}\]
and 
\[\C'=\left\{
\left(\begin{array}{c}
\scalebox{0.8}{%
\xymatrix{
& 1 \ar[dr] & \\
1 \ar[ur] & & 2 \ar[ll]
}}
\end{array}\right),
\left(\begin{array}{c}
\xymatrix{
2 \ar@<0.2pc>[r] & 3 \ar@<0.2pc>[l]
}
\end{array}\right)\right\}\]
where the first and the second cycles have multiplicity 1 and 2, respectively.

\item Let $(Q,\C)$ be the SB quiver as in Example \ref{examples of ssb} (3).
Then we get the right mutation $\mu_1^+(Q,\C)=(Q',\C')$ of $Q$ at 1 as follows:
\[\begin{CD}
Q=\fbox{$\begin{array}{c}
\xymatrix@C=1.5cm{
& 1 \ar@<-0.2pc>[d] \ar@<0.2pc>[d]  & \\
3 \ar[ru] & 2 \ar[l] \ar[r] & 4 \ar[ul] 
}
\end{array}$} @>\mbox{(QM1)}>> 
\fbox{$\begin{array}{c}
\xymatrix@C=1.5cm{
& 1 \ar@<-0.2pc>[d] \ar@<0.2pc>[d]  & \\
3 \ar[ru] \ar@<-0.2pc>[r] & 2 \ar@<-0.2pc>[l] \ar@<0.2pc>[r] & 4 \ar[ul] \ar@<0.2pc>[l] 
}
\end{array}$} \\
@. @VV\mbox{(QM2)}V \\
Q'=\fbox{$\begin{array}{c}
\xymatrix@C=1.5cm{
& 1 \ar[dl] \ar[dr]& \\
3 \ar@<-0.2pc>[r] & 2 \ar@<-0.2pc>[u] \ar@<0.2pc>[u] & 4 \ar@<0.2pc>[l] 
}
\end{array}$} @<<\mbox{(QM3)}<
\fbox{$\begin{array}{c}
\xymatrix@C=1.5cm{
& 1 \ar@<-0.2pc>[d] \ar@<0.2pc>[d]  & \\
3 \ar@<-0.2pc>[r] & 2 \ar@<-0.2pc>[l] \ar@<0.2pc>[r] & 4 \ar@<0.2pc>[l] 
}
\end{array}$}
\end{CD}\]
and 
\[\C'=\left\{
\left(\begin{array}{c}
\xymatrix@R=0.5cm{
 & 1 \ar[dl] \\
3 \ar[r] & 2 \ar[u]
}
\end{array}\right),
\left(\begin{array}{c}
\xymatrix@R=0.5cm{
1 \ar[dr] \\
2 \ar[u] & 4 \ar[l] 
}
\end{array}\right)\right\}
\]
\end{enumerate}
\end{example}

\subsection{Multiplex case (1)}

Next, we introduce mutation at multiplex vertices and its cycle-decomposition. 
They are defined by making minor alterations to mutation at non-multiplex vertices.

Let $(Q,\C)$ be a SB quiver and fix a vertex $i$ of $Q$. 
We consider the following situation:
\[\xymatrix@R=1pt{
j' \ar[dr]^{\alpha'} & & \\
 & i \ar@<0.2pc>[r]^\alpha \ar[dl]^{\beta'} & j \ar@<0.2pc>[l]^\beta\\
h & &
}\]
with $\beta\neq\na(\alpha)$ and $\alpha\neq\na(\beta)$:
in this case, it is observed that $\alpha=\na(\alpha')$ and $\beta'=\na(\beta)$.

We define a new SB quiver $\mu_i^+(Q,\C)=(Q',\C')$ as follows.

\subsubsection{Mutation rules}

\begin{definition}\label{QM multiplex}
We assume that $j'\neq h$.
A quiver $Q'$ of $Q$ at $i$ is defined by the following three steps:
\begin{enumerate}[(QM1)']
\item Draw a new arrow $\xymatrix{j' \ar[r]^x & j}$
\begin{enumerate}[(QM1-1)']
\item if $j'\neq j$ or
\item[(QM1-2)'] if $j'=j$ and $\mult(C_\alpha)>1$.
\end{enumerate}
\item Remove two arrows $\alpha$ and $\alpha'$.
\item Add new arrows in the following way:
\begin{enumerate}[(QM3-1)']
\item If there is an arrow $\gamma:\xymatrix{h \ar[r] & h'}$ with $\gamma\neq \na(\beta')$, 
then remove $\gamma$ and add new arrows $\xymatrix{h \ar[r]^x & i \ar[r]^y & h'}$.
\item Otherwise, add new arrows $\xymatrix{h \ar@<0.2pc>[r]^x & i \ar@<0.2pc>[l]^y}$.
\end{enumerate}
\end{enumerate}
\end{definition}

We can easily check that the new quiver $Q'$ is again special.

\subsubsection{Cycle-decompositions}

We give a cycle-decomposition $\C'$ of $Q'$.

\begin{definition}\label{relations of QM multiplex}
Assume that $j'\neq h$.
We use the notation of Definition \ref{QM multiplex}.
\begin{enumerate}[(1)]
\item We define a cycle containing a new arrow $x$ in (QM1)' as follows: 
\begin{enumerate}[(i)]
\item In the case (QM1-1)', $C_x$ is obtained by replacing $\alpha'\alpha$ in $C_\alpha$ by $x$.
\item In the case (QM1-2)', $C_x$ is a new cycle $\xymatrix{j\ar@(ur,dr)^x}$ with multiplicity $\mult(C_\alpha)$.
\end{enumerate}
\item We define a cycle containing a new arrow $x$ and $y$ in (QM3)' as follows:
\begin{enumerate}[(i)]
\item In the case (QM3-1)', $C_x=C_y$ and replace $\gamma$ in $C_\gamma$ by $xy$.
\item In the case (QM3-2)', $C_x$ and $C_y$ are new cycles 
satisfying $C_x=C_y=\left(\xymatrix{h \ar@<0.2pc>[r]^x & i \ar@<0.2pc>[l]^y}\right)$
with multiplicity 1.
\end{enumerate}
\end{enumerate}
Then we have a cycle-decomposition $\C'$ of $Q'$.
\end{definition}

Thus, we get a new SB quiver $\mu_i^+(Q,\C)=(Q',\C')$, called \emph{right mutation} of $(Q,\C)$ at $i$.

Dually, we define the \emph{left mutation} $\mu_i^-(Q,\C)$ of $(Q,\C)$ at $i$ 
by $\mu_i^-(Q,\C):=\mu_i^+(Q^{\rm op},\C^{\rm op})^{\rm op}$.

\begin{example}\label{example of multiplex}
Let $Q$ be the quiver 
\[\xymatrix{
 & 1 \ar[dl]_{\alpha_1} \ar@<0pc>[rd]|{\beta} & \\
2 \ar[rr]_{\alpha_2} & & 3 \ar@<0.4pc>[lu]^{\alpha_3} \ar@<-0.4pc>[lu]_{\beta'}
}\]
with cycle-decomposition $\C$:
\[Q=\left(\begin{array}{c}
\scalebox{0.8}{%
\xymatrix{
 & 1 \ar[dl]_{\alpha_1} & \\
2 \ar[rr]_{\alpha_2} & & 3 \ar[lu]_{\alpha_3}
}}
\end{array}\right)\cup 
\left(\begin{array}{c}
\scalebox{0.8}{%
\xymatrix{
1 \ar@<-0.2pc>[rd]_{\beta} & \\
 & 3  \ar@<-0.2pc>[lu]_{\beta'}
}}
\end{array}\right)
\]
such that the multiplicity of each cycle is 1.
Then we see that the right mutation $\mu_1^+(Q,\C)=(Q',\C')$ of $Q$ at 1 is 
\[\begin{CD}
Q:=\fbox{$\begin{array}{c}
\xymatrix{
 & 1 \ar[dl] \ar@<0pc>[rd] & \\
2 \ar[rr] & & 3 \ar@<0.4pc>[lu] \ar@<-0.4pc>[lu]
}
\end{array}$} @>\mbox{(QM1)'}>>
\fbox{$\begin{array}{c}
\xymatrix{
 & 1 \ar[dl] \ar@<0pc>[rd] & \\
2 \ar[rr] & & 3 \ar@<0.4pc>[lu] \ar@<-0.4pc>[lu]
}\end{array}$} \\
@. @VV\mbox{(QM2)'}V \\
Q'=\fbox{$\begin{array}{c}
\xymatrix{
 & 1 \ar@<0.4pc>[dl] \ar@<-0.4pc>[dl]  & \\
2 \ar[rr] \ar@<0pc>[ur] & & 3 \ar[lu]
}\end{array}$} @<<\mbox{(QM3)'}<
\fbox{$\begin{array}{c}
\xymatrix{
 & 1 \ar[dl]  & \\
2 \ar[rr] & & 3 \ar@<0.4pc>[lu] 
}\end{array}$}
\end{CD}\]
and
\[\C'=\left\{
\left(\begin{array}{c}
\scalebox{0.8}{%
\xymatrix{
 & 1 \ar[dl] & \\
2 \ar[rr] & & 3 \ar[lu]
}}
\end{array}\right), 
\left(\begin{array}{c}
\scalebox{0.8}{%
\xymatrix{
& 1 \ar@<0.2pc>[dl] \\
2 \ar@<0.2pc>[ur] &
}}
\end{array}\right)\right\}\]
\end{example}

\subsection{Multiplex case (2)}\label{triple multiplex}

Finally, we introduce the last case of mutation of SB quivers. 

Let $(Q,\C)$ be a SB quiver and fix a vertex $i$ of $Q$.
Suppose that the subquiver of $Q$ around the vertex $i$ is 
\[\xymatrix{
j' \ar@<0.2em>[r]^{\alpha'} & i \ar@<0.2em>[r]^\alpha \ar@<0.2em>[l]^{\beta'} & j \ar@<0.2em>[l]^{\beta}
}\]
with $\beta\neq\na(\alpha)$ and $\alpha\neq\na(\beta)$:
i.e., the case of $j'=h$ in Multiplex case (1).

\begin{definition}\label{QM multiplex 2}
We define the \emph{right mutation} $\mu_i^+(Q,\C)$ of $(Q,\C)$ at $i$ by $\mu_i^+(Q,\C)=(Q,\C)$.
\end{definition}

Dually, the \emph{left mutation} $\mu_i^-(Q,\C)$ of $(Q,\C)$ at $i$ is also defined by $\mu_i^-(Q,\C)=(Q,\C)$.


\section{Reduction theorem}\label{section of reduction}


The aim of this section is to give `reduction' theorem for symmetric special biserial algebras, 
which is a generalization of Rickard's star theorem (see Section \ref{star theorem}). 

Throughout this section, let $A$ be a symmetric special biserial algebra associated with a SB quiver $(Q,\C)$.

Note first that every $\mu_i^+$ preserves the property of being (non-) multiplex, that is, 
$A$ is non-multiplex at any vertex of $Q$ if and only if so is $\mu_i^+(A)$. 

The main theorem of this section is the following.

\begin{theorem}\label{reduction theorem}
Let $A$ be a symmetric special biserial algebra associated with 
a SB quiver $(Q, \{C_0, C_1,\cdots,C_s\})$ satisfying $\mult(C_0)\geq\mult(C_1)\geq\cdots\geq\mult(C_s)$.
Then the algebra $A$ is derived equivalent to a symmetric special biserial algebra 
with a cycle-decomposition $\{C'_0,C'_1,\cdots,C'_v\}$
satisfying the following properties:
\begin{enumerate}[{\rm (1)}]
\item All vertices appear in $C'_0$;
\item $v\leq1$ if $\mult(C_2)=1$. Otherwise, $v=\max\{\ell\in\{2,\cdots,s\}\ |\ \mult(C_\ell)\neq1 \}$;
\item $\mult(C_\ell')=\mult(C_\ell)$ for any $0\leq\ell\leq v$;
\item Each $C'_\ell$ for $\ell\geq 2$ is a loop.
\end{enumerate}
\end{theorem}

To show this, we introduce a method for reducing some cycles.

\begin{method}\label{reduction method}
Let $(Q,\C)$ be a SB quiver and $C_0$ a cycle of $Q$.
We construct a new SB quiver $(Q',\C')$ as follows.
\begin{enumerate}[(R1)]
\item Let $i$ be a vertex of $Q$ which does not belong to $C_0$ and 
assume that there exist arrows $\xymatrix{i\ar[r]^\alpha&h\ar[r]^\beta&j}$ with $C_0=C_\beta$.
Then mutating at $i$, we obtain new arrows $\xymatrix{h \ar[r]^x & i \ar[r]^y & j}$ of $\mu_i^+(Q,\C)$ with $C_x=C_y$.
Thus, we observe that the vertex $i$ belongs to $C_x$:
Again, write $C_x$ by $C_0$.

Repeating this argument, we get a SB quiver $(Q',\C')$ having a cycle to which all vertices belong. 

\item Assume that all vertices are in $C_0$.
Fix a cycle $C:\xymatrix{i_1 \ar[r]^{\gamma_1} & i_2 \ar[r]^{\gamma_2} & \cdots \ar[r]^{\gamma_{\ell-1}} & 
i_\ell \ar@/^1pc/[lll]^{\gamma_\ell}}$ ($i_1\neq i_2$)
of $Q$ which is not $C_0$.
Let $\alpha:\xymatrix{i_1\ar[r]&h}$ be an arrow in $C_0$. 
Suppose that $P_h$ is uniserial or that 
there exists an arrow $\xymatrix{h \ar[r]^\beta & h'}$ with $C_\beta\neq C_0$ and $h\neq i_\ell$.
We consider the following situation
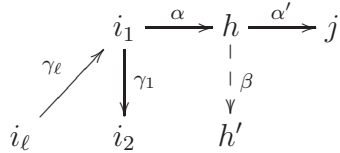
\[\xymatrix{
   & i_1 \ar[r]^\alpha \ar[d]^{\gamma_1} & h \ar[r]^{\alpha'} \ar@{-->}[d]^\beta & j \\
i_\ell \ar[ru]^{\gamma_\ell} & i_2  &    h' &
}\]
where $C_\alpha=C_{\alpha'}=C_0$.
Then mutating at $i_1$, we see that 
the vertex $i_1$ is in a new cycle $\xymatrix{i_1 \ar@<0.2pc>[r]^\alpha & h \ar@<0.2pc>[l]}$ or in $C_\beta$, 
and does not belong to $C_{\gamma_2}$.
Since the vertex $i_2$ is contained in $C_0$, 
the vertex $i_1$ belongs to the cycle $C_{\alpha'}$ of $\mu_i^+(Q)$ as well as all the other vertices:
Again, write $C_{\alpha'}$ by $C_0$.

Continuing this argument, we get a SB quiver $(Q',\C')$ satisfying the following:
\begin{enumerate}[(i)]
\item $\C'$ admits a cycle $C_0$ to which all vertices belong.
\item The vertex $i_\ell$ is only in $C_0$ if $\mult(C)=1$,
otherwise there is a loop at $i_\ell$ 
with multiplicity $\mult(C)$.
\item $\C'$ has a cycle of the form
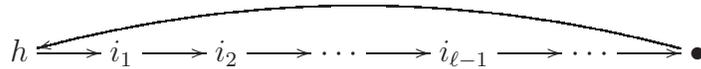
\[\xymatrix{h \ar[r] & i_1 \ar[r] & i_2 \ar[r] & \cdots \ar[r] & i_{\ell-1} \ar[r] & \cdots \ar[r] & 
\bullet \ar@/_1.5pc/[llllll]^{} }\]
with multiplicity 1 or $\mult(C_\beta)$.
\end{enumerate}

\item Assume that all vertices are in $C_0$.
Let $i$ be a vertex of $Q$ having a loop $\xymatrix{i\ar@(ur,dr)^\alpha}$ with $C_\alpha\neq C_0$.
Let $\xymatrix{h \ar[r] & i \ar[r] & j}$ be a path in $C_0$.
Then mutating at $i$ twice, we obtain a SB quiver $(Q',\C')$ satisfying the following: 
\begin{enumerate}[(i)]
\item $\C'$ admits a cycle $C_0$ to which all vertices belong.
\item There is an arrow $\xymatrix{h\ar[r]&j}$ in $C_0$.
\item There is a loop $\alpha'$ at $i$ with $C_{\alpha'}\neq C_0$. 
\end{enumerate}
\end{enumerate}
\end{method}

Now we are ready to prove Theorem \ref{reduction theorem}.

\begin{proof}[Proof of Theorem \ref{reduction theorem}.]
By applying (R1), we can assume that all vertices are in $C_0$.
Finish if $s\leq1$. 

Assume $s\geq2$ and let $C_s$ be a cycle of the form
$\xymatrix{i_1 \ar[r] & i_2 \ar[r] & \cdots \ar[r] & i_\ell \ar@/^1pc/[lll]^{}}$.

(i) If $C_s$ is not a loop, 
then we do (R2) for $C_s$:
Denote by $(Q',\C')$ the new SB quiver.
There exists a cycle $C_0'$ in $\C'$ to which all vertices belong and
it is observed that the vertex $i_\ell$ of $Q'$ is only in $C_0'$ or
there is a loop at $i_\ell$.
Moreover, there exists a cycle $C$ in $\C'$ having of the form
$\xymatrix{h \ar[r] & i_1 \ar[r] & \cdots \ar[r] & i_{\ell-1} \ar[r] & \cdots \ar[r] & \bullet \ar@/^1pc/[lllll]^{}}$.
If the number of vertices in $C$ is greater than $\ell$, then we obtain that the number of cycles in $\C'$ is at most $s$.

Assume that $C$ is just of the form
$\xymatrix{h \ar[r] & i_1 \ar[r] & \cdots \ar[r] & i_{\ell-1} \ar@/^1pc/[lll]^{}}$.
Let $i$ be the vertex of $Q'$ having the unique arrow $h\to i$ in $C_0'$.
By applying (R3) if necessary, we may suppose that there is no loop at $i$.
As $s\geq2$, we can also assume that there is an arrow $i\xrightarrow{\beta} j$ of $Q'$ with $C_\beta\neq C_0'$.
Doing (R2) for $C$, one gets a new SB quiver $(Q'', \{C_0'',\cdots,C_v''\})$ with $v\leq s$.


(ii) Let $t\geq2$.
Assume that $C_{t+1},\cdots,C_s$ are loops
but $C_t$ is not a loop.
Put $C_t:=\left\{\xymatrix{i_1 \ar[r] & i_2 \ar[r] & \cdots \ar[r] & i_\ell \ar@/^1pc/[lll]^{}}\right\}$
and let $h$ be the vertex having the unique arrow $i_1\to h$ in $C_0$.
By applying (R3) if necessary, we may suppose that $h$ does not belong to $C_{t'}$ for $t+1\leq t'\leq s$.
Then do (R2) for $C_t$ and let $(Q',\C')$ be the new SB quiver.
We see that there is a loop at $i_\ell$ in $Q'$ and $\C'$ has a cycle $C$ of the form
$\xymatrix{h \ar[r] & i_1 \ar[r] & \cdots \ar[r] & i_{\ell-1} \ar[r] & \cdots \ar[r] & \bullet \ar@/^1pc/[lllll]^{}}$.
By the same argument in (i) , we can get a SB quiver $(Q'',\{C_0'',\cdots,C_s''\})$
such that $C_u''$ is a loop for $t\leq u\leq s$.

Thus, the proof is complete.
\end{proof}

\section{Brauer graph algebras}\label{Brauer graph}

In this section, we introduce a \emph{flip} of Brauer graphs and
show that it is compatible with a tilting mutation of Brauer graph algebras.
In particular any Brauer graph algebra is symmetric special biserial (and converse also holds \cite{Ro}), 
therefore all the theorems stated in the previous sections are applied to Brauer graph algebras.


We recall the definition of Brauer graphs.

\begin{definition}
A \emph{Brauer graph} $G$ is a graph with the following data:
\begin{enumerate}[(i)]
\item There exists a cyclic ordering of the edges adjacent to each vertex, 
usually described by the clockwise ordering.
\item For every vertex $v$, there exists a positive integer $m_v$ assigned to $v$, called the \emph{multiplicity} of $v$. 
We say that a vertex $v$ is \emph{exceptional} if $m_v>1$ 
\end{enumerate}
\end{definition}

\subsection{Flip of Brauer graphs}

Let $G$ be a Brauer graph. 
For a cyclic ordering $(\cdots,i,j,\cdots)$ adjacent to a vertex $v$ with $j\neq i$,
we write $j$ by $e_v(i)$ and denote by $v_j(i)$ (simply, $v(i)$) the vertex of $j$ distinct from $v$ if it exists,
otherwise $v(i):=v$.

We say that an edge $i$ of $G$ is \emph{external} if it has a vertex with cyclic ordering which consists of only $i$, 
otherwise it is said to be \emph{internal}. 

We now introduce flip of Brauer graphs. 

\begin{definition}\label{flip of Brauer graph}
Let $G$ be a Brauer graph and fix an edge $i$ of $G$. 
We define the \emph{flip} $\mu_i^+(G)$ of $G$ as follows:

\begin{enumerate}[CASE (1)]
\item[Case (1)] The edge $i$ has the distinct two vertices $v$ and $u$:
\begin{itemize}
\item If $i$ is internal, then 
\begin{enumerate}[(Step 1)]
\item detach $i$ from $v$ and $u$;
\item attach it to $v(i)$ and $u(i)$ by $e_{v(i)}(e_v(i))=i$ and $e_{u(i)}(e_u(i))=i$, respectively.
\end{enumerate}
Locally there are the following three cases:
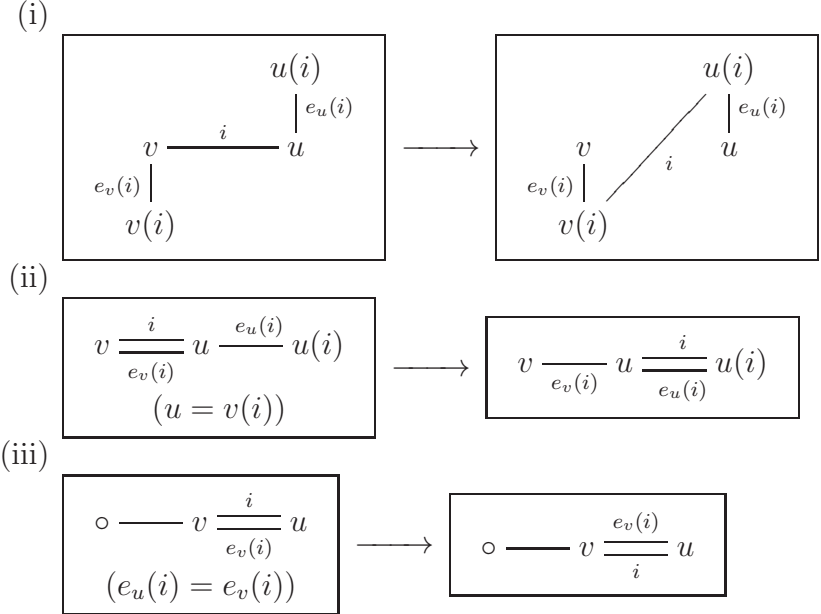
\begin{enumerate}[(i)]
\item \mbox{}

$\begin{CD}
\fbox{$\begin{array}{c}
\xymatrix@R=0.5cm @C=1cm{
        &        u(i) \\
     v \ar@{-}[r]^i \ar@{-}[d]_{e_v(i)} & u \ar@{-}[u]_{e_u(i)}    \\
v(i) &       
}
\end{array}$} @>>>
\fbox{$\begin{array}{c}
\xymatrix@R=0.5cm @C=1cm{
        &        u(i) \\
     v  \ar@{-}[d]_{e_v(i)} & u \ar@{-}[u]_{e_u(i)}    \\
v(i) \ar@{-}[ruu]_i &       
}
\end{array}$}
\end{CD}$

\item \mbox{}

$\begin{CD}
\fbox{$\begin{array}{c}
\xymatrix{
v \ar@<0.2pc>@{-}[r]^i & u \ar@{-}@<0.2pc>[l]^{e_v(i)} \ar@{-}[r]^{e_u(i)} & u(i) 
}\\
(u=v(i))
\end{array}$} @>>>
\fbox{$\begin{array}{c}
\xymatrix{
v \ar@{-}[r]_{e_v(i)} & u \ar@{-}@<0.2pc>[r]^i & u(i) \ar@{-}@<0.2pc>[l]^{e_u(i)}
} 
\end{array}$} 
\end{CD}$

\item \mbox{}

$\begin{CD}
\fbox{$\begin{array}{c}
\xymatrix{
\circ \ar@{-}[r] & v \ar@<0.2pc>@{-}[r]^i & u \ar@{-}@<0.2pc>[l]^{e_v(i)}
} \\
(e_u(i)=e_v(i))
\end{array}$} @>>>
\fbox{$\begin{array}{c}
\xymatrix{
\circ \ar@{-}[r] & v \ar@{-}@<0.2pc>[r]^{e_v(i)} & u \ar@{-}@<0.2pc>[l]^{i}
}
\end{array}$} 
\end{CD}$
\end{enumerate}

\item If $i$ is external, namely $u$ is at end, then 
\begin{enumerate}[(Step 1)]
\item detach $i$ from $v$;
\item attach it to $v(i)$ by $e_{v(i)}(e_v(i))=i$.
\end{enumerate}
The local picture is the following:
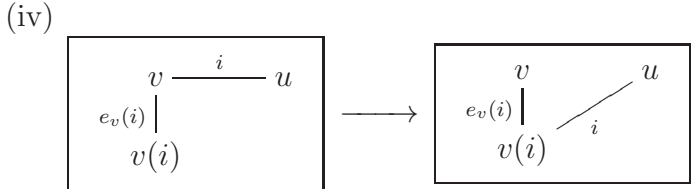
\begin{enumerate}[(iv)]
\item \mbox{}

$\begin{CD}
\fbox{$\begin{array}{c}
\xymatrix@R=0.5cm @C=1cm{
     v \ar@{-}[r]^i \ar@{-}[d]_{e_v(i)} & u    \\
v(i) &       
}
\end{array}$} @>>>
\fbox{$\begin{array}{c}
\xymatrix@R=0.5cm @C=1cm{
     v  \ar@{-}[d]_{e_v(i)} & u   \\
v(i) \ar@{-}[ru]_i &       
}
\end{array}$}
\end{CD}$
\end{enumerate}
\end{itemize}

\item[Case (2)] The edge $i$ has only one vertex $v$:
\begin{itemize}
\item If there exists the distinct two edges $h$ and $j$ written by $e_v(i)$, then 
\begin{enumerate}[(Step 1)]
\item detach $i$ from $v$;
\item attach it to $v_h(i)$ and $v_j(i)$ by $e_{v_h(i)}(h)=i$ and $e_{v_j(i)}(j)=i$.
\end{enumerate}
Locally there are the following two cases:
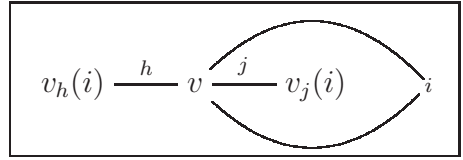
\begin{enumerate}[(i)]
\item[(v)] \mbox{}

$\begin{CD}
\fbox{$\begin{array}{c}
\\
\xymatrix{
v_h(i) \ar@{-}[r]^(0.6)h & v \ar@{-}[r]^(0.4)j \ar@{}@/^2pc/[rr] \ar@{}[rr]|(1)i   & v_j(i) & \ar@{}@/^2pc/[ll] 
} \\
\mbox{}
\end{array}$} \\
@VVV \\
\fbox{$\begin{array}{c}
\\
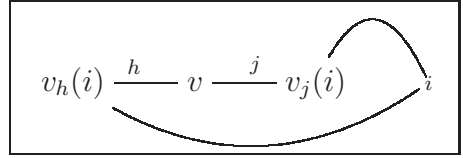
\xymatrix{
v_h(i) \ar@{-}[r]^h & v \ar@{-}[r]^j & v_j(i) \ar@{-}@/^2pc/[r] \ar@{}[r]|(1)i & \ar@{-}@/^2pc/[lll]
}
\\
\\
\end{array}$}
\end{CD}$

\item[(vi)] \mbox{}

$\begin{CD}
\fbox{$\begin{array}{c}
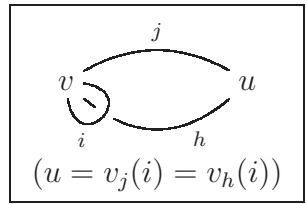
\xymatrix{
v \ar@{-}@(r,d)^(0.7)i \ar@{-}@/^1pc/[rr]^j \ar@{-}@/_1.5pc/[rr]|(0.21)\hole_(0.7)h    & & u 
}\\
(u=v_j(i)=v_h(i))
\end{array}$} @>>>
\fbox{$\begin{array}{c}
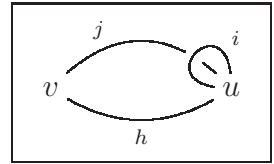
\xymatrix{
v \ar@{-}@/^1.5pc/[rr]|(0.79)\hole^(0.3)j \ar@{-}@/_1pc/[rr]_h    & & u \ar@{-}@(u,l)_(0.2)i
}
\end{array}$}
\end{CD}$
\end{enumerate}

\item Otherwise,
\begin{enumerate}[(Step 1)]
\item detach $i$ from $v$;
\item attach it to the only one vertex $v(i)$ by $e_{v(i)}(e_v(i))=i$.
\end{enumerate}
The local picture is the following:
\begin{enumerate}[(i)]
\item[(vii)] \mbox{}

$\begin{CD}
\fbox{$\begin{array}{c}
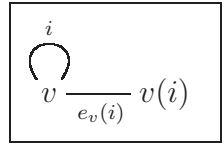
\xymatrix{
v \ar@{-}[r]_(0.45){e_v(i)} \ar@{-}@(ul,ur)^i & v(i)
}
\end{array}$} @>>>
\fbox{$\begin{array}{c}
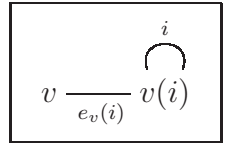
\xymatrix{
v \ar@{-}[r]_(0.45){e_v(i)}  & v(i) \ar@{-}@(ul,ur)^i
}
\end{array}$}
\end{CD}$
\end{enumerate}

\end{itemize}
\end{enumerate}
In all cases, the multiplicity of any vertex does not change.

Dually, we define $\mu_i^-(G)$ by $\mu_i^-(G):=\left(\mu_i^+(G^{\rm op})\right)^{\rm op}$
where the opposite Brauer graph, namely its cyclic ordering is described by counter-clockwise,
is denoted by $G^{\rm op}$.
\end{definition}

Every case of flip of Brauer graphs is covered in Definition \ref{flip of Brauer graph}.

We also point out that our flip of Brauer graphs can be regarded as a generalization of flip of triangulations 
of surfaces \cite{FST, MS}.

\begin{example}\label{examples of flip}
For a Brauer graph, we denote by $\bullet$ an exceptional vertex and 
by $\circ$ a non-exceptional vertex.
\begin{enumerate}[(1)]
\item Let $G$ be the Brauer graph 
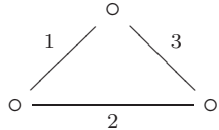
\[\xymatrix{ 
& \circ \ar@{-}[dl]_1 \ar@{-}[dr]^3 & \\
\circ \ar@{-}[rr]_2 & & \circ
}\]
Then the flip of $G$ at 1 is 
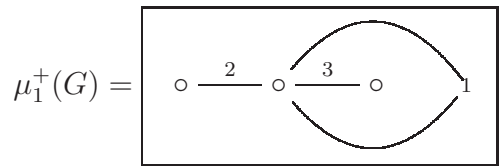
\[\mu_1^+(G)=\fbox{$\begin{array}{c}
\\
\xymatrix{
\circ \ar@{-}[r]^2 & \circ \ar@{-}[r]^3 \ar@/^2pc/@{-}[rr]\ar@{-}@/_2pc/[rr]\ar@{}[rr]|(1)1 & \circ & 
}\\[1.7em]
\end{array}$}\]

\item  Let $G$ be the Brauer graph
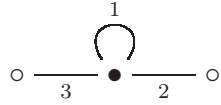
\[\xymatrix{
\circ \ar@{-}[r]_3 & \bullet \ar@{-}[r]_2 \ar@{-}@(ul,ur)^1 & \circ 
}\]
such that the multiplicity of the exceptional vertex $\bullet$ is 2.
Then we have the flip of $G$ at 1:
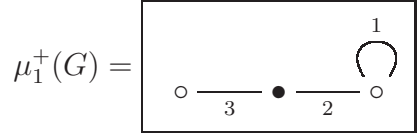
\[\mu_1^+(G)=\fbox{$\begin{array}{c}
\xymatrix{
\circ \ar@{-}[r]_3 & \bullet \ar@{-}[r]_2  & \circ \ar@{-}@(ul,ur)^1
}\end{array}$}\]

\item Let $G$ be the Brauer graph
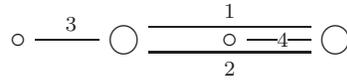
\[\xymatrix{
\circ \ar@{-}[r]^3 & \bigcirc \ar@{-}@<0.4pc>[rr]^1 & \circ \ar@{-}[r]|4 & \bigcirc \ar@{-}@<0.4pc>[ll]^2
}\]
Then the flip of $G$ at 1 is observed by
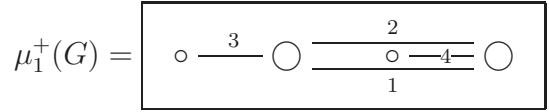
\[\mu_1^+(G)=\fbox{$\begin{array}{c}
\xymatrix{
\circ \ar@{-}[r]^3 & \bigcirc \ar@{-}@<0.4pc>[rr]^2 & \circ \ar@{-}[r]|4 & \bigcirc \ar@{-}@<0.4pc>[ll]^1
}\end{array}$}\]

\item Let $G$ be the Brauer graph  
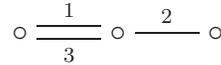
\[\xymatrix{
\circ \ar@<0.2pc>@{-}[r]^1 & \circ \ar@<0.2pc>@{-}[l]^3 \ar@{-}[r]^2 & \circ
}\]
Then the flip of $G$ at 1 is:
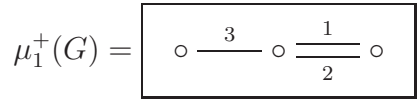
\[\mu_1^+(G)=\fbox{$\begin{array}{c}
\xymatrix{
\circ \ar@{-}[r]^3 & \circ \ar@<0.2pc>@{-}[r]^1 & \circ \ar@<0.2pc>@{-}[l]^2 
}\end{array}$}\]
\end{enumerate}
\end{example}

\subsection{Compatibility of flip and tilting mutation}

For a Brauer graph $G$, we denote by $\vx(G)$ the set of the vertices of $G$.

We construct a SB quiver from a Brauer graph.

\begin{definition}
Let $G$ be a Brauer graph.
A \emph{Brauer quiver} $Q=Q_G$ is a finite quiver given by a Brauer graph $G$ as follows:
\begin{enumerate}[(i)]
\item There exists a one-to-one correspondence between vertices of $Q$ and edges of $G$.
\item For two distinct edges $i$ and $j$, 
an arrow $\xymatrix{i\ar[r]&j}$ of $Q$ is drawn if there exists a cyclic ordering of the form $(\cdots,i,j,\cdots)$.
\item For an edge $i$ of $G$, we draw a loop at $i$ if it has an exceptional vertex which is at end.
\end{enumerate}
Then $Q$ is special. 

For each vertex $v$ of $G$, let $(i_1,i_2,\cdots,i_s,i_1)$ be a cyclic ordering at $v$.
Then we define a cycle $C_v$ by 
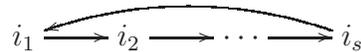
\[\xymatrix{
i_1 \ar[r] & i_2 \ar[r] & \cdots \ar[r] & i_s \ar@/_1pc/[lll]}\]
with multiplicity $m_v$ if $s\neq1$, otherwise by an empty set.
We have a cycle-decomposition $\C=\C_G=\{C_v\ |\ v\in\vx(G)\}$.

Thus we obtain a SB quiver $(Q,\C)$.
\end{definition}

For a Brauer graph $G$,
a \emph{Brauer graph algebra} $A=A_G$ is a symmetric special biserial algebra associated with the SB quiver $(Q_G,\C_G)$.

It is known that the notion of Brauer graph algebras is nothing but that of symmetric special biserial algebras.
The following result is obtained.

\begin{proposition}\label{SSBA is BGA}
\begin{enumerate}[{\rm (1)}]
\item \cite{Ro,An} An algebra is a Brauer graph algebra if and only if it is symmetric special biserial.
\item The property of being a Brauer graph algebra is derived invariant.
\end{enumerate}
\end{proposition}
\begin{proof}
The second assertion follows from the first assertion and Proposition \ref{ssb under derived}.
\end{proof}



For an edge $i$ of a Brauer graph $G$, 
we say that $G$ \emph{has multi-edges} at $i$ if there exists a subgraph 
$\xymatrix{\circ \ar@{}[r]^(0)v \ar@{}[r]^(1)u \ar@<0.2pc>@{-}[r]^i & \circ \ar@{-}@<0.2pc>[l]^j}$ of $G$ such that 
the cyclic orderings at $v$ and $u$ are $(\cdots,i,j,\cdots)$ and $(\cdots,j,i,\cdots)$, respectively;
it is allowed that $u=v$.

We have the following easy observation. 

\begin{proposition}
Let $G$ be a Brauer graph and $i$ be an edge of $G$. 
Then $Q_G$ is multiplex at $i$ if and only if $G$ has multi-edges at $i$.
\end{proposition}

It is not difficult to see that flip of each Brauer graph $G$ coincides with right mutation of 
the corresponding SB quiver $(Q_G, \C_G)$, that is,
we have:

\begin{proposition}\label{mutation and flip}
Let $G$ be a Brauer graph and $i$ be an edge of $G$. 
Then one has $(Q_{\mu_i^+(G)},\C_{\mu_i^+(G)})=\mu_i^+(Q_G,\C_G)$.
\end{proposition}

We observe that Example \ref{examples of QM} (1)--(3) and Example \ref{example of multiplex} coincide with 
Example \ref{examples of flip} (1)--(3) and (4), respectively.

Applying Theorem \ref{compatibility} to Brauer graph algebras,
we figure out that flip of Brauer graph is compatible with tilting mutation of Brauer graph algebras.

\begin{theorem}\label{main Brauer graph}
Let $G$ be a Brauer graph and $i$ be an edge of $G$. 
\begin{enumerate}[{\rm (1)}]
\item We have an isomorphism $A_{\mu_i^+(G)}\simeq \mu_i^+(A_G)$.
\item The algebra $\mu_i^+(A_G)$ has the Brauer graph $\mu_i^+(G)$.
\item The algebra $A_{\mu_i^+(G)}$ is derived equivalent to $A_G$.
\end{enumerate}
\end{theorem}
\begin{proof}
We have isomorphisms
\[A_{\mu_i^+(G)} \simeq A_{(Q_{\mu_i^+(G)},\C_{\mu_i^+(G)})} 
               \stackrel{\ref{mutation and flip}}{\simeq} A_{\mu_i^+(Q_G,\C_G)} 
               \stackrel{\ref{compatibility}}{\simeq} \mu_i^+(A_{(Q_G,\C_G)})
               \simeq \mu_i^+(A_G).
\]
The second and the last assertion follow immediately.
\end{proof}

\begin{remark}\label{remark of Kauer}
Special cases of this theorem were given in \cite{K}, 
where he considered the cases (i)(iv) and (vii) in Definition \ref{flip of Brauer graph}.
\end{remark}


\subsection{(Double-) Star theorem}\label{star theorem}

In this subsection, we apply Theorem \ref{reduction theorem} to Brauer graph algebras 
and finally obtain a generalization of Rickard's star theorem. 

Let $G$ be a Brauer graph.
We denote by $\m_G$ the sequence $(m_{v_1},\cdots,m_{v_\ell})$ of the multiplicities of 
all vertices satisfying $m_{v_1}\geq\cdots\geq m_{v_\ell}$.

For a Brauer graph algebra $A$,
the Brauer graph of $A$ is denoted by $G_A$. 

A Brauer graph $G$ is said to be \emph{double-star}
if there exist two vertices $v$ and $u$ of $G$ such that any edge is either of the following:
\begin{itemize}
\item It is external having the vertex $v$:
\item It has both the vertices $v$ and $u$:
\item It has only the vertex $v$, that is, it is of the form $\xymatrix{v \ar@{-}@<0.4pc>@(ru,rd)^{} }$.
\end{itemize}
We call $v$ and $u$ \emph{center} and \emph{vice-center}, respectively.

We say that a Brauer double-star $G$ satisfies \emph{multiplicity condition}
if the multiplicities of the center and the vice-center are the first and the second greatest 
among them of all vertices of $G$, respectively.

As a consequence of Theorem \ref{reduction theorem}, the following theorem is obtained. 

\begin{theorem}\label{reduction of Brauer graph}
Any Brauer graph algebra $A$ is derived equivalent to a Brauer graph algebra with 
Brauer double-star $G$ satisfying multiplicity condition such that
\begin{enumerate}[{\rm (i)}]
\item the number of the edges of $G$ coincides with that of $G_A$ and
\item $\m_G=\m_{G_A}$, in particular $G$ and $G_A$ have the same number of exceptional vertices.
\end{enumerate}
\end{theorem}
\begin{proof}
Let $\Lambda$ be a symmetric special biserial algebra with a cycle-decomposition $\C:=\{C_0',\cdots,C_v'\}$ 
as in Theorem \ref{reduction theorem} and put $G:=G_\Lambda$.

We first show that $G$ is a Brauer double-star.
As the condition (1) in Theorem \ref{reduction theorem},
the cycle $C_0'$ corresponds to the center of $G$.
We see that the vice-center of $G$ is given by the cycle $C_1'$.
Considering also the condition (4) of Theorem \ref{reduction theorem},
we have that $G$ is a Brauer double-star.

By the condition (3) of Theorem \ref{reduction theorem},
it is obtained that $G$ satisfies multiplicity condition.
Since $A$ and $\Lambda$ are derived equivalent, 
they have the same number of non-isomorphic simple modules,
whence the condition (i) is satisfied.
As the condition (2) and (3) of Theorem \ref{reduction theorem},
it follows that $G$ satisfies the condition (ii). 
\end{proof}

We raise a question on classification of derived equivalence classes of Brauer graph algebras.

\begin{question}
For a given Brauer graph $G$,
is there up to isomorphism and opposite isomorphism 
a unique Brauer double-star algebra satisfying the multiplicity condition and which is derived equivalent to the algebra $A_G$?
\end{question}

It is well-known that this question has a positive answer if $G$ is a tree as a graph.
Such a Brauer graph is said to be a \emph{generalized Brauer tree}.
It is called \emph{Brauer tree} if it has at most one exceptional vertex. 
A \emph{(generalized) Brauer star} is a (generalized) Brauer tree and a Brauer double-star. 
Note that any edge of a generalized Brauer star is external
and every vertex can be a vice-center.

From Theorem \ref{reduction of Brauer graph}, 
we deduce star theorem for generalized Brauer tree algebras.
 
\begin{corollary}\label{Rickard reduction theorem}\cite{R2,MH}
\begin{enumerate}[{\rm (1)}]
\item Any generalized Brauer tree algebra $A$ is derived equivalent to a generalized Brauer star algebra $B$
with $\m_{G_B}=\m_{G_A}$ such that the multiplicity of the center is maximal. 
\item Derived equivalence classes of generalized Brauer tree algebras are determined by 
the number of the edges and the multiplicities of the vertices.
\end{enumerate}
\end{corollary} 




\section{A proof of main theorem}\label{proof}

In this section we prove Theorem \ref{compatibility}. 

Our proof of this theorem consists of three steps:
The first is to decide the shape of the quiver of the algebra $\mu_i^+(A)$.
The second is to give relations of $\mu_i^+(A)$, 
which is done by only determining every cycle.
The last is to furnish the multiplicities to each cycle.

\subsection{Preliminaries}

We prepare two important tools for our proof of the main theorem.

To determine the structure of a finite dimensional algebra $A$, 
the data $\dim\Hom_A(P_i,P_j)$ for projective $A$-modules $P_i$ and $P_j$, 
namely the \emph{Cartan data} of $A$,
play an important role. 

The following proposition gives the Cartan data of a derived equivalent algebra to $A$. 

\begin{proposition}\label{Cartan}\cite{H}
Let $A$ be a finite dimensional algebra. 
Let $T$ and $U$ be bounded complexes of projective $A$-modules 
satisfying $\Hom_{\Db(\mod A)}(T,U[n])=0$ for any integer $n\neq0$.
Then the following equality holds:
\[\dim\Hom_{\Db(\mod A)}(T,U)=\sum_{\ell, m}(-1)^{\ell+m}\dim\Hom_A(T^\ell,U^m)\]
where the $n$-th term of a complex $V$ is denoted by $V^n$.
\end{proposition}

We denote by $\smod A$ the stable category of $\mod A$.
It is well-known that $\smod A$ is triangulated category if $A$ is self-injective. 
Moreover, a derived equivalence between two self-injective algebras yields a stable equivalence 
between themselves \cite{R2}. 

For an Okuyama-Rickard complex $T(i)$ of a symmetric algebra $A$, 
we denote by $F_i:\smod A\to\smod \mu_i^+(A)$ the stable equivalence between $A$ and $\mu_i^+(A)$ 
induced by $T(i)$.

The syzygy of any module is denoted by $\Omega$.

Then we can exactly describe a $\mu_i^+(A)$-module $F_i(X)$ for every $A$-module $X$. 

\begin{proposition}\cite{O}(see also \cite[Lemma 3.4]{A2})\label{simples}
Let $A=kQ/I$ be a symmetric algebra and fix a vertex $i$ of $Q$. 
Then the following gives a complete list of simple $\mu_i^+(A)$-modules:
\begin{enumerate}[{\rm (1)}]
\item $F_i(\Omega S_i)$;
\item $F_i(S_j)$ if there is no such arrow $i\to j$ of $Q$;
\item $F_i(X_j)$ where $X_j$ is maximal among submodules $X$ of $P_j$ such that 
only $S_i$ appears in a composition factor of $X/S_j$,
if there is an arrow $i\to j$ of $Q$.
\end{enumerate}
\end{proposition}
%
%
%
%

\subsection{Extensions among simple modules}\label{ext}

The shape of the quiver of a given algebra is determined by the dimensions of extensions among simple modules.
This subsection is devoted to giving them for the endomorphism algebra of an Okuyama-Rickard complex
of a symmetric special biserial algebra.

Throughout this subsection, assume that $A=kQ/I$ is a symmetric special biserial algebra and fix a vertex $i$ of $Q$. 
We put $B:=\mu_i^+(A)$ and still denote by $S_j$ a simple $B$-module corresponding to $T_j$ in Definition-Theorem \ref{OR complex}. 

To calculate the dimensions of extensions among simple $B$-modules, 
we use Proposition \ref{simples} and feel free to utilize four formulas
\begin{eqnarray*}
\Ext_A^1(X,Y)\simeq\Ext_B^1(F_iX,F_iY); &
\Ext_A^1(X, Y)\simeq\sHom_A(\Omega X, Y); \\
\sHom_A(X, S)=\Hom_A(X, S); &
\sHom_A(S, X)=\Hom_A(S,X) 
\end{eqnarray*}

for any $A$-modules $X, Y$ and simple $A$-module $S$.

Let $j$ and $j'$ be vertices of $Q$ distinct from $i$. 
\begin{enumerate}[(1)]
\item
\begin{enumerate}[(i)]
\item If there is neither arrows $\xymatrix{i\ar[r]&j}$ nor $\xymatrix{i\ar[r]&j'}$,
then we observe an isomorphism $\Ext_B^1(S_j,S_{j'})\simeq \Ext_A^1(S_j, S_{j'})$.
\item If there is no arrow $\xymatrix{i\ar[r]&j}$ but is an arrow $\xymatrix{i\ar[r]&j'}$,
then we obtain isomorphisms 
\[\def\arraystretch{1.3}
\begin{array}{rl}
\Ext_B^1(S_j,S_{j'}) &\simeq \Ext_A^1(S_j,X_{j'}) \\
                     &\simeq \sHom_A(S_j, \Omega^{-1}X_{j'})\\
                     &= \Hom_A(S_j, \Omega^{-1}X_{j'}).
\end{array}\]
This implies that $\dim\Ext_B^1(S_j,S_{j'})$ is given by
\[\begin{cases}
\ 1 & \mbox{if there is an arrow } \xymatrix{j\ar[r]&j'}\ \mbox{of }Q; \\
\ 1 & \mbox{if there is a cycle containing } \xymatrix{j \ar[r] & i \ar[r] & j'}\ \mbox{of }Q;\\\\
\ 1 & \mbox{if there is a cycle containing } \xymatrix{j \ar[r] & i \ar[r]\ar@(ul,ur) & j'}\ \mbox{of }Q;\\
\ 2 & \mbox{if there are two cycles containing } \xymatrix{j \ar[r] & i \ar[r] & j'}\ \mbox{of }Q;\\
\ 0 & \mbox{otherwise}.
\end{cases}\]
\item If there is an arrow $\xymatrix{i\ar[r]&j}$ but is no arrow $\xymatrix{i\ar[r]&j'}$,
then we see isomorphisms
\[\def\arraystretch{1.3}
\begin{array}{rl}
\Ext_B^1(S_j,S_{j'}) &\simeq \Ext_A^1(X_j, S_{j'}) \\
                     &\simeq \sHom_A(\Omega X_j,S_{j'}) \\
                     &= \Hom_A(\Omega X_j,S_{j'}).
\end{array}\]
Thus we get that $\dim\Ext_B^1(S_j,S_{j'})=1$ if there exists a cycle containing $\xymatrix{i\ar[r]&j\ar[r]&j'}$ of $Q$ 
and no double arrow $\xymatrix{i \ar@<0.2pc>[r]\ar@<-0.2pc>[r] & j}$, otherwise it is zero. 
\item Assume that there are both arrows $\xymatrix{i\ar[r] & j}$ and $\xymatrix{i\ar[r] & j'}$. 
Then we observe isomorphisms 
$\Ext_B^1(S_j,S_{j'})\simeq \Ext_A^1(X_j,X_{j'}) \simeq \sHom_A(X_j,\Omega^{-1}X_{j'})$.
Since we have no non-zero homomorphism from $S_i$ to $\Omega^{-1}X_{j'}$, 
any non-zero homomorphism $X_j\to \Omega^{-1}X_{j'}$ is monomorphism. 
Therefore all projective homomorphisms from $X_j$ to $\Omega^{-1}X_{j'}$ are zero, 
and so we see an equality $\sHom_A(X_j,\Omega^{-1}X_{j'})=\Hom_A(X_j,\Omega^{-1}X_{j'})$. 
Thus we obtain $\dim\Ext_B^1(S_j,S_{j'})$ as follows:
\begin{itemize}
\item In the case of $j\neq j'$:
\[\begin{cases}
\ 1 & \mbox{if there is a cycle containing } \xymatrix{i \ar[r] & j \ar[r] & i \ar[r] & j'}\ \mbox{of }Q;\\
\ 1 & \mbox{if there is a cycle containing } \xymatrix{i \ar[r] & j \ar[r] & j'}\ \mbox{of }Q;\\
\ 0 & \mbox{otherwise}.
\end{cases}\]
\item In the case of $j=j'$:
\[\begin{cases}
\ 1 & \mbox{if there is a cycle } \xymatrix{i \ar@<0.2pc>[r] & j \ar@<0.2pc>[l]}\ \mbox{or }
      \xymatrix{i \ar@<0.2pc>[r] \ar@(dl,ul)^{} & j \ar@<0.2pc>[l]}\ \mbox{of }Q \\ 
    & \mbox{with multiplicity exactly greater than 1}; \\
\ 1 & \mbox{if there is a cycle containing } \xymatrix{i \ar[r] & j \ar@(ru,rd)^{}}\ \mbox{of }Q; \\
\ 0 & \mbox{otherwise}.
\end{cases}\]
\end{itemize}
\end{enumerate}

\item 
\begin{enumerate}[(i)]
\item Assume that there is no arrow $\xymatrix{i\ar[r]&j}$. 
\begin{enumerate}[(a)]
\item We have isomorphisms 
\[\Ext_B^1(S_j,S_i)\simeq \Ext_A^1(S_j,\Omega S_i) 
                   \simeq \sHom_A(S_j,S_i)=0.\]
\item We obtain isomorphisms
\[\def\arraystretch{1.3}
\begin{array}{rl}
\Ext_B^1(S_i,S_j) &\simeq \Ext_A^1(\Omega S_i,S_j) \\
                  &\simeq \sHom_A(\Omega^2S_i,S_j) \\
                  &=\Hom_A(\Omega^2S_i,S_j),
\end{array}\]
which implies that $\dim\Ext_B^1(S_i,S_j)$ coincides with the number of paths 
$\xymatrix{i \ar[r]^\alpha & h \ar[r]^\beta & j}$ of $Q$ with $\alpha\beta\in I$.
\end{enumerate}
\item Assume that there is an arrow $\xymatrix{i\ar[r]&j}$.
\begin{enumerate}[(a)]
\item We observe isomorphisms
\[\def\arraystretch{1.3}
\begin{array}{rl}
\Ext_B^1(S_j,S_i) &\simeq \Ext_A^1(X_j,\Omega S_i) \\
                  &\simeq \sHom_A(X_j, S_i) \\
                  &= \Hom_A(X_j,S_i),
\end{array}\]
and so one sees that $\dim\Ext_B^1(S_j,S_i)=\dim\Ext_A^1(S_i,S_j)$.
\item We have isomorphisms
\[\def\arraystretch{1.3}
\begin{array}{rl}
\Ext_B^1(S_i,S_j) &\simeq \Ext_A^1(\Omega S_i,X_j) \\
                  &\simeq \sHom_A(S_i,\Omega^{-2}X_j) \\
                  &=\Hom_A(S_i,\Omega^{-2}X_j).
\end{array}\]
Thus we get the following equalities.
\begin{itemize}
\item[(b1)] In the case that $P_j$ is uniserial:
\[\dim\Ext_B^1(S_i,S_j)=\begin{cases}
\ 2 & \mbox{if there is a path } \xymatrix{h \ar@<0.2pc>[r]^\beta & i \ar@<0.2pc>[l]^\alpha \ar[r]^\gamma & j}\ 
      \mbox{of }Q\\ & \mbox{with } \alpha\beta\in I\ \mbox{and } \beta\gamma\not\in I;\\
\ 1 & \mbox{otherwise}.
\end{cases}\]
\item[(b2)] In the case that $P_j$ is non-uniserial:
$\dim\Ext_B^1(S_i,S_j)$ is the number of paths 
(b2-1) $\xymatrix{i \ar[r]^\alpha & h \ar[r]^\beta & j}$ and
(b2-2) $\xymatrix{h \ar@<0.2pc>[r]^\beta & i \ar@<0.2pc>[l]^\alpha \ar[r]^\gamma & j}$
with $\alpha\beta\in I$ and $\beta\gamma\not\in I$.
\end{itemize}
\end{enumerate}
\end{enumerate}

\item We have isomorphisms $\Ext_B^1(S_i,S_i)\simeq \Ext_A^1(\Omega S_i,\Omega S_i)\simeq \Ext_A^1(S_i,S_i)$.
\end{enumerate}

\subsection{Proof of our theorem}

Our theorem is proved dividing into three cases as well as the definition of mutation of SB quivers.

\subsubsection{Non-multiplex case}

We show Theorem \ref{compatibility} in the case of non-multiplex.
Assume that $A=kQ/I$ is a symmetric special biserial algebra and is non-multiplex at $i$.
We use the notation of Definition \ref{QM} and Definition-Theorem \ref{OR complex}. 

\begin{proof}
(1) We first show that the quiver of $\mu_i^+(A)$ coincides with $\mu_i^+(Q,\C)$.
\begin{enumerate}[(QM1)]
\item The calculations \ref{ext} (1)(ii) and (1)(iv) yield the new arrows of (QM1-1) and (QM1-2), respectively. 

\item[(QM3)] By \ref{ext} (2)(ii)(a), we reverse arrows $\xymatrix{i \ar[r] & j}$. 
Using  \ref{ext} (2)(i)(b) and (2)(ii)(b2-1), we get the new arrows of (QM3-1).
To obtain the new arrows of (QM3-2), 
we need to consider two cases: 
If $P_j$ is uniserial, then they occur by \ref{ext} (2)(ii)(b1).
If $P_j$ is non-uniserial, then we have the new arrows of (QM3-2) by considering $h=j$ and $\gamma=\alpha$ in \ref{ext} (2)(ii)(b2-2).
\end{enumerate}

It is observed that the other arrows do not change. 

(2) We show that the cycle-decomposition of $\mu_i^+(A)$ is given by Definition \ref{relations of QM}. 
Note that $\mu_i^+(A)$ is symmetric special biserial by Lemma \ref{ssb under derived}.
\begin{enumerate}[(QM1)]
\item We consider a new arrow $x$ as in (QM1). 
One defines a morphism $T_j\to T_h$ corresponding to $x$ as follows:
\[x:P_j\xrightarrow{\beta}P_i\xrightarrow{\alpha}P_h\ \mbox{or } 
P_j\xrightarrow{\beta}P_i\xrightarrow{\gamma}P_i\xrightarrow{\alpha}P_h.\]
\begin{enumerate}[(QM1-1)]
\item Assume that $h\neq j$. 
If an arrow $y$ satisfies $y\alpha\not\in I$ or $\beta y\not\in I$,
then we obtain $C_y=C_\alpha$ or $C_y=C_\beta$.
This implies that $yx\neq0$ or $xy\neq0$ in $\mu_i^+(A)$, 
which means that $C_x$ is given by replacing $\alpha\beta$ by $x$.
\item[(QM1-2)] Assume that $h=j, \beta\alpha\not\in I$ and $\mult(C_\alpha)>1$. 
Then we can check $x^2\neq0$ in $\mu_i^+(A)$, which implies that $C_x$ has only one arrow $x$. 
\end{enumerate}
\item[(QM3)] We consider new arrows $x$ and $y$ as in (QM3). 
\begin{enumerate}[(QM3-1)]
\item In the case (QM3-1), we define morphisms $T_i\to T_h$ and $T_j\to T_i$ corresponding to $x$ and $y$, respectively:
\[\xymatrix{
T_i: \ar@<-0.4pc>[d]_x & P_h\oplus P' \ar[r]^(0.6){[\alpha,f]} \ar[d]_{[1,0]} & P_i \ar[d] & & 
T_j: \ar@<-0.4pc>[d]_y & P_j \ar[r] \ar[d]_{\left[\begin{smallmatrix}\beta\\ 0\end{smallmatrix}\right]} & 0 \ar[d] \\
T_h: & P_h \ar[r] & 0 & \mbox{and} & 
T_i: & P_h\oplus P' \ar[r]_(0.6){[\alpha,f]} & P_i
}\]
Then we see $xy=\beta$,
which implies that $C_x=C_y$ and these are obtained by replacing $\beta$ in $C_\beta$ by $xy$.
\item In the case (QM3-2), we have to consider two cases: 
One is that there exists an arrow $\xymatrix{h\ar[r]^\beta&i}$, 
and the other is that there is no such arrow. 
To avoid confusion, we write the arrow $\alpha$ of $\mu_i^+(Q,\C)$ by $\alpha'$. 

(i) Assume that there exists an arrow $\xymatrix{h\ar[r]^\beta&i}$, 
and then we have arrows $\xymatrix{i \ar@<0.2pc>[r]^\alpha & h \ar@<0.2pc>[l]^\beta \ar[r]^\gamma & j}$ of $Q$ 
with $\alpha\gamma\not\in I$.
Since $A$ is special biserial, we observe that $\alpha\beta$ belongs to $I$.
By our assumption, we obtain that $\beta\alpha\not\in I$.
We define morphisms $T_i\to T_h$ and $T_h\to T_i$ corresponding to $x$ and $\alpha'$ of $\mu_i^+(Q,\C)$
as follows:
\[\xymatrix{
T_i: \ar@<-0.4pc>[d]_x & P_h\oplus P' \ar[r]^(0.6){[\alpha,f]} \ar[d]_{[1,0]} & P_i \ar[d] & & 
T_h: \ar@<-0.4pc>[d]_{\alpha'} & P_h \ar[r] \ar[d]_{\left[\begin{smallmatrix}\beta\alpha\\0\end{smallmatrix}\right]} & 0 \ar[d] \\
T_h: & P_h \ar[r] & 0 & \mbox{and} & 
T_i: & P_h\oplus P' \ar[r]_(0.6){[\alpha,f]} & P_i
}\]
We can easily check $x\alpha'\neq0$ and $\alpha' x=0$ in $\mu_i^+(A)$. 
Moreover, an arrow $\delta$ with $\delta\beta\not\in I$ satisfies $\delta x\neq0$ in $\mu_i^+(A)$. 
Thus we see that $C_x$ is given by replacing $\beta$ by $x$. 

(ii) Assume that there is no such arrow $\xymatrix{h\ar[r]^\beta&i}$.
Then $P_h$ is uniserial.

(a) If there is no loop at $i$ in $C_\alpha$, then we define morphisms $T_i\to T_h$ and $T_h\to T_i$ 
corresponding to $x$ and $\alpha'$ of $\mu_i^+(Q,\C)$ as follows:
\[\xymatrix{
T_i: \ar@<-0.4pc>[d]_x & P_h\oplus P' \ar[r]^(0.6){[\alpha,f]} \ar[d]_{[1,0]} & P_i \ar[d] & & 
T_h: \ar@<-0.4pc>[d]_{\alpha'} & P_h \ar[r] \ar[d]_{\left[\begin{smallmatrix}\pi\\0\end{smallmatrix}\right]} & 0 \ar[d] \\
T_h: & P_h \ar[r] & 0 & \mbox{and} & 
T_i: & P_h\oplus P' \ar[r]_(0.6){[\alpha,f]} & P_i
}\]
where $\pi$ is the composition of the canonical epimorphism $P_h\to S_h$ and the canonical inclusion $S_h\to P_h$.
It is easy to see that $x\alpha'=\pi\neq0$. 
Let $\gamma$ be the unique arrow of $Q$ starting at $h$ with $\alpha\gamma\not\in I$. 
We have $\alpha'\gamma=0$ in $\mu_i^+(A)$.
Since the number of arrows of $\mu_i^+(Q,\C)$ starting at $h$ is at most two,
we obtain $\alpha' x\neq0$.
Thus, $C_x$ is a new cycle $\xymatrix{i \ar@<0.2pc>[r]^{\alpha'} & h \ar@<0.2pc>[l]^x}$. 

(b) Assume that there is a loop $\beta$ at $i$ in $C_\alpha$. 
Then we observe that $\beta^2\in I$ and that the complex $T_i$ has of the form $P_h\oplus P_h\xrightarrow{[\alpha,\beta\alpha]}P_i$.
To avoid confusion, we write $\beta$ in $\mu_i^+(Q,\C)$ by $\beta'$.
We define morphisms $T_i\to T_h$, $T_h\to T_i$ and $T_i\to T_i$ corresponding to $x$, $\alpha'$ and $\beta'$ as follows:
\[\xymatrix{
T_i: \ar@<-0.4pc>[d]_x & P_h\oplus P_h \ar[r]^(0.6){[\alpha,\beta\alpha]} \ar[d]_{[0,1]} & P_i \ar[d] & & 
T_h: \ar@<-0.4pc>[d]_{\alpha'} & P_h \ar[r] \ar[d]_{\left[\begin{smallmatrix}\pi\\0\end{smallmatrix}\right]} & 0 \ar[d] \\
T_h: & P_h \ar[r] & 0 & , & 
T_i: & P_h\oplus P_h \ar[r]_(0.6){[\alpha,\beta\alpha]} & P_i
}\]
\[\xymatrix{
T_i: \ar@<-0.4pc>[d]_{\beta'} & P_h\oplus P_h \ar[r]^(0.6){[\alpha,\beta\alpha]} 
\ar[d]_{\left[\begin{smallmatrix}0&0\\ 1&0\end{smallmatrix}\right]} & P_i \ar[d]^\beta \\
T_i: & P_h\oplus P_h \ar[r]_(0.6){[\alpha,\beta\alpha]} & P_i 
}\]
Then we get $x\beta'\alpha'=\pi$.
It follows from the same argument above that $\alpha' x\neq0$ in $\mu_i^+(A)$. 
Thus we see that $C_x$ is a new cycle 
$\begin{array}{c}\xymatrix@C=0.5cm @R=0.5cm{
 & i \ar[dr]^{\alpha'} & \\ i \ar[ur]^{\beta'} & & h \ar[ll]^x 
}\end{array}$.
\end{enumerate}
\end{enumerate}

(3) Finally, we give the multiplicity of each cycle of $\mu_i^+(Q,\C)$. 
Note that it is determined by the Cartan data of $\mu_i^+(A)$. 
To do this, we use the formula of Proposition \ref{Cartan}. 
Since every indecomposable projective $A$-module $P_j$ distinct from $P_i$ is a direct summand of $T(i)$,
the Cartan datum $\dim\Hom_{\Db(\mod A)}(T_h,T_j)$ for $h\neq i$ and $j\neq i$ is equal to $\dim\Hom_A(P_h,P_j)$.
Thus, we have only to show that the multiplicity of a cycle as in Definition \ref{relations of QM} (2)(ii)(b) is 1.
Put $m:=\mult(C_\alpha)$.

Assume that there is no loop at $i$ in $C_\alpha$.
Let $\beta:\xymatrix{i\ar[r]&j}$ be an arrow of $Q$ distinct from $\alpha$ if it exists.
Then the complex $T_i$ has of the form $P_h\oplus P_j\xrightarrow{[\alpha,\beta]}P_i$ if $\beta$ exists, 
otherwise $T_h\xrightarrow{\alpha}P_i$.
\begin{itemize}
\item If $\beta$ exists and $C_\beta=C_\alpha$, 
then we have equalities
\[\def\arraystretch{1.3}
\begin{array}{rl}
\dim\Hom_{\Db(\mod A)}(T_h,T_i) &= \dim\Hom_A(P_h,P_h)+\dim\Hom_A(P_h,P_j) \\ 
                                &\ \ \ -\dim\Hom_A(P_h,P_i) \\
                                &= (m+1)+ m -2m \\
                                &= 1.
\end{array}\]
\item Otherwise,
\[\def\arraystretch{1.3}
\begin{array}{rl}
\dim\Hom_{\Db(\mod A)}(T_h,T_i) &= \dim\Hom_A(P_h,P_h)-\dim\Hom_A(P_h,P_i) \\
                                &= (m+1)-m \\
                                &= 1.
\end{array}\]
\end{itemize}

If there is a loop $\beta$ at $i$ of $Q$ in $C_\alpha$, 
then the complex $T_i$ has of the form $P_h\oplus P_h\xrightarrow{[\alpha,\beta\alpha]}P_i$. 
We obtain equalities
\[\def\arraystretch{1.3}
\begin{array}{rl}
\dim\Hom_{\Db(\mod A)}(T_h,T_i) &= 2\cdot\dim\Hom_A(P_h,P_h)-\dim\Hom_A(P_h,P_i) \\
                                &= 2(m+1)-2m \\
                                &= 2.
\end{array}\]

These give that the multiplicity of our cycle is 1.
\end{proof}

\subsubsection{Multiplex case (1)}

We prove Theorem \ref{compatibility} in the case of multiplex.
Assume that $A=kQ/I$ is a symmetric special biserial algebra and is multiplex at $i$. 
We use the notation of Definition \ref{QM multiplex} and Definition-Theorem \ref{OR complex}. 
Suppose that $j'\neq h$.

\begin{proof}
(1) We first show that the quiver of $\mu_i^+(A)$ coincides with $\mu_i^+(Q,\C)$.
\begin{enumerate}[(QM1)']
\item The calculations \ref{ext} (1)(ii) and (1)(iv) yield the new arrows of (QM1-1)' and (QM1-2)', respectively. 
However, \ref{ext} (1)(iv) says that there is no new arrow from $j$ to $h$.
\item The vertex $j$ belongs to exactly two cycles: One is $C_\alpha$ and the other is $C_\beta$. 
By \ref{ext} (2)(ii)(b2), we have no arrow from $i$ to $j$, which means that we remove $\alpha$. 
If $j'\neq j$, by \ref{ext} (2)(i)(a), there is no new arrow from $j'$ to $i$.
If $j'=j$, by \ref{ext} (2)(ii)(a) and $h\neq j$, the number of new arrows from $j$ to $i$ is 1.
Therefore we remove $\alpha'$.
\item The new arrows of (QM3-1)' are obtained by the same argument in the proof to non-multiplex case. 
We consider new arrows of (QM3-2)', which occurs if $P_h$ is uniserial: Note that $h\neq j$.
It follows from \ref{ext} (2)(ii)(b1) that the number of arrows from $i$ to $h$ is exactly two: One is $\beta'$ and 
the other is a new arrow $y$. 
By \ref{ext} (2)(ii)(a), we have exactly one arrow from $h$ to $i$, which is a new arrow $x$.
\end{enumerate}

It is observed that the other arrows do not change. 

(2) We show that the cycle-decomposition of $\mu_i^+(A)$ is given by Definition \ref{relations of QM multiplex}. 
Note that $\mu_i^+(A)$ is symmetric special biserial by Lemma \ref{ssb under derived}.
By considering the same argument in the proof to non-multiplex case, 
we have only to give morphisms among the direct summands of $T$ corresponding to $\beta, \beta'$ and $x,y$ of (QM3-2)'. 
Note that $T_i$ has of the form $P_j\oplus P_h\xrightarrow{[\alpha,\beta']}P_i$. 

To avoid confusion, we write the arrows $\beta$ and $\beta'$ of $\mu_i^+(Q,\C)$ by $\beta_*$ and $\beta'_*$, respectively. 
We define morphisms of $T_i\to T_j$ and $T_h\to T_i$ corresponding to $\beta_*$ and $\beta'_*$ as follows:
\[\xymatrix{
T_i: \ar@<-0.4pc>[d]_{\beta_*} & P_j\oplus P_h \ar[r]^(0.6){[\alpha,\beta']} \ar[d]_{[1,0]} & P_i \ar[d] & & 
T_h: \ar@<-0.4pc>[d]_{\beta'_*} & P_h \ar[r] \ar[d]_{\left[\begin{smallmatrix}\beta\beta'\\0\end{smallmatrix}\right]} & 0 \ar[d] \\
T_j: & P_j \ar[r] & 0   & \mbox{and} & 
T_i: & P_j\oplus P_h \ar[r]_(0.6){[\alpha,\beta']} & P_i 
}\]

Assume that $P_h$ is uniserial. 
Then we define morphisms $T_i\to T_h$ and $T_h\to T_i$ corresponding to $x$ and $y$ as follows:
\[\xymatrix{
T_i: \ar@<-0.4pc>[d]_{x} & P_j\oplus P_h \ar[r]^(0.6){[\alpha,\beta']} \ar[d]_{[0,1]} & P_i \ar[d] & & 
T_h: \ar@<-0.4pc>[d]_{y} & P_h \ar[r] \ar[d]_{\left[\begin{smallmatrix}0\\\pi\end{smallmatrix}\right]} & 0 \ar[d] \\
T_h: & P_h \ar[r] & 0   & \mbox{and} & 
T_i: & P_j\oplus P_h \ar[r]_(0.6){[\alpha,\beta']} & P_i 
}\]

(3) We get the multiplicity of each cycle by the same argument in the proof to non-multiplex case. 

Thus the proof is complete.
\end{proof}

\subsubsection{Multiplex case (2)}

Finally we have to show Theorem \ref{compatibility} in the case of subsection \ref{triple multiplex}.
However it follows from the same argument in the proof to non-multiplex and multiplex cases.


\end{document}